\documentclass[a4paper,oneside,10.5pt]{article}%
\usepackage{amsmath}
\usepackage{amsfonts}
\usepackage{indentfirst}
\usepackage{amssymb}
\usepackage{graphicx}%
\usepackage{appendix}
\usepackage[colorlinks,citecolor=blue,urlcolor=blue]{hyperref}
\usepackage[utf8]{inputenc}
\usepackage{float}

\providecommand{\U}[1]{\protect \rule{.1in}{.1in}}

\newtheorem{theorem}{Theorem}[section]

\newtheorem{corollary}[theorem]{Corollary}

\newtheorem{assumption}[theorem]{Assumption}
\newtheorem{example}[theorem]{Example}

\newtheorem{lemma}[theorem]{Lemma}

\newtheorem{remark}[theorem]{Remark}

\newenvironment{proof}[1][Proof]{\noindent \textbf{#1.} }{\  \rule{0.5em}{0.5em}}
\numberwithin{equation}{section}

\begin{document}

\title{Minimum-Time Stochastic Optimal Control Problems Under Mean Constraints and Application to Portfolio Investment}
\author{Shuzhen Yang\thanks{Shandong University-Zhong Tai Securities Institute for Financial Studies, Shandong University, PR China, (yangsz@sdu.edu.cn).}}
\date{}
\maketitle

\textbf{Abstract}: Motivated by the practical demand for minimum-time optimal investment problems, we develop a unified framework for mean constraints minimum-time stochastic optimal control problems. In this setting, the minimum-time criterion is defined as the expected earliest time to reach a target state, making the terminal time dependent on the control. The main contributions of this work are twofold: first, we derive an extended stochastic maximum principle for the proposed model and further prove the existence of an optimal control for the linear systems; second, we establish a bang-bang type optimal control for the linear time-optimal control problem. In the end, we solve a financial portfolio optimization problem within the proposed framework.


{\textbf{MSC2010}: 93E03; 93E20; 60G99

\addcontentsline{toc}{section}{\hspace*{1.8em}Abstract}

\section{Introduction}

Dynamic portfolio optimization with uncertain exit time is an important research topic in mathematical finance and stochastic control theory. In financial markets, the exit time of investment activities is not fixed or purely exogenous, but is closely dependent on the real-time state and cumulative performance of investment portfolios. Existing literature has extensively explored portfolio optimization problems with stochastic exit times. Specifically, \cite{CNM08} investigated optimal investment strategies under constant relative risk aversion (CRRA) utility preferences, where the uncertain exit time is characterized by a given probability density function. The study established sufficient optimality conditions and derived explicit analytical solutions for optimal investment strategies. Nevertheless, such horizon frameworks cannot reflect the linkage between portfolio performance and investment termination decisions.

A series of studies have supplemented the research on portfolio optimization with state-dependent exit times. \cite{MU05} constructed a static mean-variance portfolio framework with return-dependent exit times, while \cite{H08} proposed a robust portfolio selection strategy based on the worst-case conditional value-at-risk (CVaR) risk measure for uncertain horizon scenarios. Despite these advances, most existing studies focus on static portfolio decisions or stochastic horizon settings, lacking a unified dynamic stochastic control framework for performance-triggered stopping problems.

Classical stochastic optimal control theory typically considers fixed terminal time settings. For a given fixed terminal time \(T\), the standard cost functional consists of a running cost term defined on the time interval \([0,T]\) and a terminal cost term at the maturity moment, formulated as follows:
\begin{equation}
\label{incos-1}
J(u(\cdot))=
\mathbb{E}\bigg{[}\displaystyle\int_0^Tf(X^u(t),u(t))\mathrm{d}t+\Psi(X^u(T))\bigg{]},
\end{equation}
where the state process \(X^u(\cdot)\) is governed by a controlled stochastic differential equation (SDE). This fixed terminal time control structure has been widely adopted to model traditional portfolio optimization problems. For fundamental theories and applications of stochastic optimal control in mathematical finance, readers may refer to the monographs \cite{FM06,Y99}. For classic stochastic maximum principle theories, \cite{B78,B81} established local maximum principles for stochastic control systems, while \cite{P90} further developed the global maximum principle for general control domains.

Subsequent research has continuously expanded the boundary of stochastic optimal control theory. \cite{17H,HJX18,W13,Y10} investigated optimal control problems for recursive utility systems and fully coupled forward-backward stochastic differential systems. \cite{16HJ} extended the stochastic maximum principle to recursive optimal control systems with volatility ambiguity. In terms of infinite-dimensional stochastic control systems, \cite{LZ14} studied stochastic maximum principles for backward stochastic evolution equations, and \cite{QT12} explored optimality conditions for quasi-linear backward stochastic partial differential equations. Additionally, \cite{F10,R17,BE10} examined optimal control problems with state constraints, and \cite{Y18} established necessary and sufficient conditions for stochastic systems with multi-time state cost functionals. For mean-field stochastic control systems, \cite{AD11,Li12} proposed local maximum principles, and \cite{BDL11,BLM16} further generalized the global maximum principle for general mean-field-type dynamic systems.

Time-optimal control theory constitutes another critical research branch closely related to this study. In deterministic control systems, time-optimal control aims to determine the minimum time for the state trajectory to reach a preset target domain, where the terminal stopping time is determined by the control strategy. When the running cost is normalized to 1 and the terminal cost is zero, the classical time-optimal control problem can be equivalently formulated as finding the minimum time to reach the target constraint. The standard mathematical formulation is given by
$$
\tau^u=\inf\bigg{\{}t:Q(X^u(t))\leq 0,\ t\in [0,T] \bigg{\}}\wedge T.
$$
The theoretical foundation of deterministic time-optimal control was laid by \cite{P62}, who established the fundamental maximum principle for time-optimal control problems. Comprehensive introductions to classical optimal control theory can be found in \cite{K04}, while \cite{L94} extended time-optimal control analysis to infinite-dimensional systems. \cite{E21} further developed adjoint solution methods for numerical analysis of time-optimal control problems.

For stochastic time-optimal control with varying terminal times, existing research remains relatively limited. \cite{Y20} first proposed a stochastic optimal control framework with minimum-time mean constraints, defining the varying terminal time as
$$
\tau^u=\inf\bigg{\{}t:\mathbb{E}[Q(X^u(t))]\leq 0,\ t\in [0,T] \bigg{\}}\wedge T,
$$
and derived the corresponding local stochastic maximum principle. On this basis, \cite{SY25} extended the result to non-convex control domains and established the global maximum principle for varying-terminal-time systems. \cite{WY24} further explored recursive stochastic optimal control with minimum-time. In financial applications, \cite{Y22} constructed a mean-variance portfolio model with varying terminal times and provided optimal strategies and numerical simulation results.

However, the existing mean-constrained time-optimal frameworks only rely on the first-order moment of the state process to define stopping rules, failing to incorporate higher-order statistical characteristics and cumulative state information, which limits their applicability to complex financial investment scenarios.
To address the above research gaps, we consider to introduce a general process  $Y^u(\cdot)$ satisfying a mean-field type stochastic differential equation in this study, for $t\in [0,T]$, with the initial condition  $Y^u(0)=y_0>0$,
\begin{equation}
\label{insde-2}
\begin{array}
[c]{rl}
\text{d}Y^u(t)=h(\mathbb{E}[X^u(t)],X^u(t),\mathbb{E}[b(X^u(t),u(t))],u(t))\mathrm{d}t
+g(\mathbb{E}[X^u(t)],X^u(t),u(t))\mathrm{d}W(t).
\end{array}
\end{equation}
In this present paper, we develop a stochastic optimal control structure subject to a minimum time constraint. The corresponding cost functional is given by:
\begin{equation}
\label{incos-2}
J(u(\cdot))=\mathbb{E}\bigg{[}\displaystyle\int_0^{\tau^u}f(X^u(t),u(t))\mathrm{d}t
+\Psi(X^u(\tau^u))\bigg{]},
\end{equation}
with the minimum time constraint
\begin{equation}
\label{intime-2}
\tau^u=\inf\bigg{\{}t:\mathbb{E}[Y^u(t)]\leq 0,\ t\in [0,T] \bigg{\}}\wedge T.
\end{equation}
This formulation includes state constraint functions such as $Q(X^u(t),\mathbb{E}[X^u(t)])$ as a particular case of $\mathbb{E}[Y^u(t)]$, $t\in [0,T]$ (For further details, refer to Example \ref{exm:cases} in Subsection \ref{subsec:time-cases}). Given a control $u(\cdot)$, $Y^u(\cdot)$ can be viewed as an observable target. Notice that in (\ref{intime-2}), we define the minimum time $\tau^u$ such that
$\mathbb{E}[Y^u(\tau^u)]\leq 0,\ \tau^u\in [0,T]$ under control $u(\cdot)$. When the set
$\bigg{\{}t:\mathbb{E}[Y^u(t)]\leq 0,\ t\in [0,T] \bigg{\}}$ is nonempty, the constraint condition is equivalent to requiring
$$
\mathbb{E}[Y^u(t)]< 0, \ t\in [0,\tau^{u}), \ \mathbb{E}[Y^u(\tau^{u})]= 0.
$$

This unified stochastic optimal control framework offers significant practical utility across multiple domains where decision-making must balance a primary cost objective with a time-sensitive. For instance, in a financial market, \( X^u(\cdot) \) could represent the value of a risky asset, while \( Y^u(\cdot) \) models a key market state variable. An investment strategy would then terminate at the earliest time \( \tau^u \) that the expected state \(\mathbb{E}[Y^u(\cdot)]\) reaches a predefined threshold, thereby integrating return objectives with timing and risk constraints (See Section \ref{sec:app} for further details).
We derive necessary conditions for an optimal control $\bar{u}(\cdot)$ and its corresponding terminal time $\tau^{\bar{u}}$ by analyzing the minimum properties of both $\tau^{\bar{u}}$ and the cost functional $J(\bar{u}(\cdot))$. The associated adjoint equations are introduced, and a unified maximum principle is established. Furthermore, we prove the existence of an optimal control for linear systems, and obtain a bang-bang type optimal control for the related time-optimal control problem.

Our optimal control structure under minimum time constraint offers several advantages.

(i). We introduce two states, $X^u(\cdot)$ and $Y^u(\cdot)$, in this general framework: $X^u(\cdot)$ denotes the value of the {control state}, while $Y^u(\cdot)$ describes the value of the {target}.
We consider two distinct objectives for the states $X^u(\cdot)$ and $Y^u(\cdot)$ within this general framework: first, minimize the time to achieve the target for the average state; second, minimize the cost functional $J(u(\cdot))$.

(ii). The proposed optimal control structure, first considered in a stochastic setting, presents a novel approach even for deterministic systems under minimum-time constraints. Not only does our model consider both time and cost minimization, but it also provides a systematic framework for minimizing the cost functional while explicitly adhering to the goal of reducing time.

(iii). This general minimum-time optimal control structure is applicable to portfolio investment problems where the exit time is uncertain and determined by the expected value of portfolio characteristic. A detailed solution of this class of examples is presented in Section \ref{sec:app}.

The remainder of this paper is organized as follows. Section \ref{sec:form} formulates the generalized mean-constrained minimum-time stochastic optimal control problem and introduces basic model notations and standing assumptions. Section \ref{sec:SMP} establishes the stochastic maximum principle for the proposed unified control framework. Section \ref{sec:existence} proves the existence of optimal controls for linear stochastic systems. Section \ref{sec:time} analyzes the bang-bang property of optimal controls for linear time-optimal control problems. Section \ref{sec:app} applies the theoretical results to portfolio optimization problems and presents numerical examples to verify the effectiveness of the proposed framework. Section \ref{sec:con} concludes the paper and discusses future research directions.

\section{Formulation}\label{sec:form}

Let $(\Omega,\mathcal{F},P;\{ \mathcal{F}(t)\}_{t\geq 0})$ be a complete filtered probability space, and  $W$ be a $d$-dimensional standard Brownian motion, where $\{ \mathcal{F}(t)\}_{t\geq0}$ is the natural filtration generated by $W$ under the $P$-augmentation. Given the terminal time $T>0$, consider the following controlled stochastic differential equation:
\begin{equation}\label{ode_1}
\left\{
\begin{aligned}
&\text{d}{X}^u(t)=b(X^u{(t)},u(t))\text{d}t+\sigma (X^u{(t)},u(t))\text{d}W(t) ,\quad t\in(0,T],\\
&X(0)=x_0,
\end{aligned}
\right.
\end{equation}
where
$$
u(\cdot)\in \mathcal{U}[0,T]:=\left\{\phi(\cdot):\ \mathbb{E}\big{[}\int_0^{T}\left|\phi(t) \right|^2\mathrm{d}t \big{]}<+\infty,\ \phi(\cdot)\in\{\mathcal{F}_t\}_{t\geq 0}\right\}
$$
and $U$ is a convex subset of $\mathbb{R}^k$ for a given positive integer $k$.

In this study, we consider the cost functional
\begin{equation}
J(u(\cdot))=%
\mathbb{E}\bigg{[}{\displaystyle \int \limits_{0}^{\tau^{u}}}
f(X^u{(t)},u(t))\text{d}t+\Psi(X^u(\tau^u))\bigg{]},\label{cost-1}%
\end{equation}
subject to the minimum time constraint
\begin{equation}
\label{time-1}
\tau^u=\inf\bigg{\{}t:\mathbb{E}[Y^u(t)]\leq 0,\ t\in [0,T] \bigg{\}}\wedge T,
\end{equation}
where $Y^u(t)$ satisfies
\begin{equation}\label{ode_2}
\left\{
\begin{aligned}
&\text{d}{Y}^u(t)=h(\mathbb{E}[X^u(t)],X^u(t),\mathbb{E}[b(X^u(t),u(t))],u(t))\text{d}t
+g(\mathbb{E}[X^u(t)],X^u(t),u(t))\text{d}W(t),\\
&Y^u(0)=y_0.
\end{aligned}
\right.
\end{equation}
Note that if $y_0\leq 0$, we have $\tau^u=0$, which leads to a trivial problem. We introduce the following notations:
\[%
\begin{array}
[c]{l}%
b:\mathbb{R}^m\times U\to \mathbb{R}^{m\times 1},\\
\sigma:\mathbb{R}^{m}\times U\to \mathbb{R}^{m\times d},\\
f:\mathbb{R}^m\times U\to \mathbb{R},\\
\Psi:\mathbb{R}^{m}\to \mathbb{R},\\
h:\mathbb{R}^m\times\mathbb{R}^m\times \mathbb{R}^m\times U\to \mathbb{R},\\
g:\mathbb{R}^m\times\mathbb{R}^{m}\times U\to \mathbb{R}^{1\times d}.
\end{array}
\]
Let $\sigma=(\sigma^1,\sigma^2,\cdots,\sigma^d)$, with $\sigma^j\in \mathbb{R}^{m\times 1}$ for $j=1,2,\cdots, d$. Here, "$\top$" stands the transpose of a vector or matrix.

In the following, we show the basic assumptions through out this paper. We assume that $b,\sigma,f,h$ and $g$ are uniformly continuous and satisfy the following conditions.
\begin{assumption}
\label{ass-b}There exists a constant $c>0$ such that%
\[%
\begin{array}
[c]{c}%
\left| b(x_{1},u)-b(x_{2},u)\right| +\left| \sigma(x_{1},u)-\sigma(x_{2},u)\right|
 \leq c\left|x_1-x_2 \right|,\\
\end{array}
\]
$\forall(x_{1},u),(x_{2},u)\in{\mathbb{R}^m}\times U$.
\end{assumption}

\begin{assumption}
\label{assb-b2} There exists a constant $c>0$ such that
\[
\left|b(x,u)\right|+\left|\sigma(x,u)\right| \leq c(1+\mid x \mid),\quad \forall(x,u)\in{\mathbb{R}^m}\times U,
\]
and
\[
\left|h(x,x',x'',u)\right|+\left|g(x,x',u)\right| \leq c(1+\mid x\mid+\mid x' \mid+\mid x'' \mid),\quad \forall(x,x',x'',u)\in{\mathbb{R}^m}\times \mathbb{R}^m \times \mathbb{R}^m\times U.
\]
\end{assumption}

\begin{assumption}
\label{ass-fai}The functions $b(x,u),\sigma(x,u),f(x,u)$ are differentiable in $(x,u)$, with the derivatives of $b(x,u)$ and $\sigma(x,u)$ bounded, and the derivative of $f(x,u)$ having linear growth in $(x,u)$; The functions $h(x,x',x'',u),\ g(x,x',u)$ are differentiable in $(x,x',x'',u)$, with derivatives having linear growth in $(x,x',x'',u)$; The function $\Psi(x)$ is twice differentiable in $x$, with bounded derivatives.
\end{assumption}

Let Assumptions \ref{ass-b} and \ref{assb-b2} hold. Then there exists a unique
solution $(X^u(\cdot),Y^u(\cdot))$ to equations (\ref{ode_1}) and (\ref{ode_2}) (see \cite{LS78}). A control $\bar{u}(\cdot)\in \mathcal{U}[0,\tau^{\bar{u}}]$
satisfying
\begin{equation}
J(\bar{u}(\cdot))= \underset{u(\cdot)\in\mathcal{U}[0,\tau^u]}{\inf}J(u(\cdot)) \label{cost-2}%
\end{equation}
is called an optimal control. The corresponding state trajectory $(\bar{u}(\cdot),\bar{X}(\cdot))$ is called an optimal state trajectory or optimal pair under minimum time $\tau^{\bar{u}}$.

Under Assumptions \ref{ass-b}, \ref{assb-b2} hold, equation (\ref{ode_2}) yields
\begin{equation}
\label{mean-0}
\mathbb{E}[Y^u(s)]=y_0+\int_0^s\mathbb{E}[h(\mathbb{E}[X^u(t)],X^u(t),\mathbb{E}[b(X^u(t),u(t))],u(t))]\mathrm{d}t,\ s\in [0,T].
\end{equation}
To simplify notation, we write
$$
G^{{u}}(t):=\mathbb{E}[h(\mathbb{E}[X^u(t)],X^u(t),\mathbb{E}[b(X^u(t),u(t))],u(t))],\ t\in [0,T],
$$
so that equation (\ref{mean-0}) becomes
\begin{equation}
\label{mean-1}
\mathbb{E}[Y^u(s)]=y_0+\int_0^sG^{{u}}(t)\mathrm{d}t,\ s\in [0,T].
\end{equation}
We now introduce the following assumption on the function $G^{\bar{u}}(\cdot)$, which plays a crucial role in establishing the main results.
\begin{assumption}
\label{ass-h} Let $\tau^{\bar{u}}$ be a Lebesgue point of $G^{\bar{u}}(\cdot)$ and $G^{\bar{u}}(\tau^{\bar{u}})\neq 0$. That is, $G^{\bar{u}}(\cdot)$ is measurable at point $\tau^{\bar{u}}$.
\end{assumption}

\begin{remark}
The proof of Lemma \ref{le2-e0} requires the condition $G^{\bar{u}}(\tau^{\bar{u}})\neq 0$ in Assumption \ref{ass-h}. If $G^{\bar{u}}(\tau^{\bar{u}})=0$, counterexample can be constructed where $\lim_{\rho\to 0}\frac{\tau^{\bar{u}}-\tau^{u^{\rho}}}{\rho}$ does not exist. To address this, one may introduce an $\varepsilon$-approximate optimal control formulation, replacing $G^u(\cdot)$ with  $G^u(\cdot)+\varepsilon$ in the equation of $Y^u(\cdot)$.
\end{remark}

The primary contribution of this paper lies in deriving the stochastic maximum principle for the proposed model. The relevant proof can be found in Section \ref{sec:SMP}.
\begin{theorem}
\label{the-Max} Let Assumptions \ref{ass-b}, \ref{assb-b2}, \ref{ass-fai} and \ref{ass-h} hold,
$(\bar{u}(\cdot),\bar{X}(\cdot))$ be an optimal pair of (\ref{cost-2}). Then, there exists  $(p(\cdot),q(\cdot))$ satisfying the first-order adjoint equations (\ref{prin-1}), and the following holds:

 (i). If $\tau^{\bar{u}}<T$, then for any $u\in U$, $\text{a.e.}\ t \in[0,\tau^{\bar{u}})$, $P-\text{a.s.}$
\begin{equation}%
\begin{array}
[c]{ll}%
&H_u(\bar{X}(t),\bar{u}(t),p(t),q(t))(u-\bar{u}(t))-\displaystyle\frac{ \mathbb{E}[\tilde{\Psi}^{\bar{u}}(\tau^{\bar{u}})
+f(\bar{X}(\tau^{\bar{u}}),\bar{u}(\tau^{\bar{u}}))]\hat{K}(t)(u-\bar{u}(t))}{G^{\bar{u}}
(\tau^{\bar{u}})}\leq 0,\\
\end{array}
\label{prin-2}%
\end{equation}
where $\hat{K}(t)$ is given in Lemma \ref{le-dual} and
$$
\tilde{\Psi}^{\bar{u}}(\tau^{\bar{u}})=\Psi_x(\bar{X}(\tau^{\bar{u}}))^{\top}b(\bar{X}(\tau^{\bar{u}}),\bar{u}(\tau^{\bar{u}}))
+\frac{1}{2}\sum_{j=1}^d\sigma^j(\bar{X}(\tau^{\bar{u}}),\bar{u}(\tau^{\bar{u}}))^{\top}\Psi_{xx}(\bar{X}(\tau^{\bar{u}}))
\sigma^j(\bar{X}(\tau^{\bar{u}}),\bar{u}(\tau^{\bar{u}}));
$$

(ii). If $\bigg{\{}t:\mathbb{E}[Y^{\bar{u}}(t))]\leq 0,\ t\in [0,T] \bigg{\}}=\varnothing$, then for any $u\in U$, $\text{a.e.}\ t \in[0,\tau^{\bar{u}})$, $P-\text{a.s.}$
\begin{equation}%
\begin{array}
[c]{ll}%
&H_u(\bar{X}(t),\bar{u}(t),p(t),q(t))(u-\bar{u}(t))\leq 0;\\
\end{array}
\end{equation}

(iii). If $\inf\bigg{\{}t:\mathbb{E}[Y^{\bar{u}}(t))]\leq 0,\ t\in [0,T] \bigg{\}}=T$, then for any $u\in U$, $\text{a.e.}\ t \in[0,\tau^{\bar{u}})$, $P-\text{a.s.}$
\begin{equation}\label{equa-smp-2}
\begin{array}
[c]{ll}%
&H_u(\bar{X}(t),\bar{u}(t),p(t),q(t))(u-\bar{u}(t))-\displaystyle\frac{ \mathbb{E}[\tilde{\Psi}^{\bar{u}}(\tau^{\bar{u}})
+f(\bar{X}(\tau^{\bar{u}}),\bar{u}(\tau^{\bar{u}}))]\hat{K}(t)(u-\bar{u}(t))}{G^{\bar{u}}
(\tau^{\bar{u}})}\leq 0,\\
\end{array}
\end{equation}
or
\begin{equation}\label{equa-smp-3}
\begin{array}
[c]{ll}%
&H_u(\bar{X}(t),\bar{u}(t),p(t),q(t))(u-\bar{u}(t))\leq 0.\\
\end{array}
\end{equation}
\end{theorem}

\begin{remark}
In Theorem \ref{the-Max}, for $\tau^{\bar{u}}<T$, the necessary condition (\ref{prin-2}) consists of two parts $N_1$ and $N_2$, where
\[
\begin{cases}
N_1(t):=H_u(\bar{X}(t),\bar{u}(t),p(t),q(t))(u-\bar{u}(t)),\\
N_2(t):=\displaystyle\frac{ \mathbb{E}[\tilde{\Psi}^{\bar{u}}(\tau^{\bar{u}})
+f(\bar{X}(\tau^{\bar{u}}),\bar{u}(\tau^{\bar{u}}))]\hat{K}(t)(u-\bar{u}(t))}{G^{\bar{u}}
(\tau^{\bar{u}})}.
\end{cases}
\]
The term $N_1(\cdot)$ is used to verify the necessary condition for the optimal control in the classical stochastic optimal control problem, while the term $N_2(\cdot)$ is used to verify the necessary condition for the optimal control in the time-optimal control problem. In our unified optimal control framework, we can balance these two necessary conditions.

The application of the proposed model to portfolio investment problems is given in Subsection \ref{sub-sec:app}. A solve example is shown in Section \ref{sec:app}. Our model (\ref{cost-1})--(\ref{ode_2}) contains the standard optimal control problem as special case, deterministic time-optimal control problems, stochastic time-optimal control problems, and traditional stochastic optimal control problems with fixed terminal times. For more details,  refer to Subsections {\ref{sub-sec:deter}}, {\ref{subsec:time-cases}}, and {\ref{sub-sec:comp-stoch}}.
\end{remark}
\subsection{Application to portfolio investment problems}\label{sub-sec:app}
In financial markets, dynamic portfolio optimization under uncertain investment horizons is one of the most practically relevant and theoretically important problems. In conventional portfolio selection models, investors aim to maximize expected utility or return while the investment horizon is typically fixed or stochastic. However, real-world investment behaviors indicate that the effective exit time is strongly dependent on the running performance of the portfolio. Investors tend to terminate investment once the portfolio achieves a predefined target or fails to meet expected average performance standards. To characterize such investment mechanisms, we propose a novel class of mean-constrained minimum-time optimal stochastic control models \eqref{ode_1}--\eqref{ode_2}, which can precisely capture dynamic portfolio strategies with performance-dependent uncertain exit times.

In our model, the state process \(X^u(\cdot)\) represents the real-time wealth process of the risky portfolio, the control variable \(u(\cdot)\) corresponds to the dynamic asset allocation strategy assigned to risky assets, and the auxiliary state process $Y^u(\cdot)$ is introduced to measure the deviation of cumulative performance from the preset investment target. The uncertain investment terminal time \(\tau^u\) is defined as
\[\tau^u = \inf \left\{ t \in [0,T] : \mathbb{E}\big[Y^u(t)\big] \leq 0 \right\} \land T,\]
which implies that the investment terminates whenever the expected performance deviation becomes non-positive. This mean-dependent stopping mechanism is particularly consistent with institutional investment practices, where fund managers and institutional investors typically evaluate portfolio performance based on statistical average outcomes rather than individual sample-path fluctuations.

Existing studies have widely investigated portfolio optimization with stochastic exit times. In particular, \cite{CNM08} examined optimal investment problems under CRRA utility, where the uncertain exit time is exogenously specified via a given probability density. That work established sufficient optimality conditions and derived explicit investment strategies. However, such exogenous horizon frameworks are restrictive, since the exit time is independent of portfolio dynamics and fails to capture the endogenous linkage between investment performance and termination decisions.

In contrast, our model differs from conventional frameworks. First, the exit time \(\tau^u\) is endogenously determined by the expected portfolio performance, yielding an economically interpretable and performance-driven stopping rule. Second, we formulate a mean-constrained minimum-time stochastic control problem that jointly optimizes portfolio returns and target-reaching time, whereas most existing models only consider utility maximization under fixed or exogenous random horizons. Third, the introduced mean-field structure \(\mathbb{E}[Y^u(t)]\) enhances the robustness of optimal strategies against market randomness.

Accordingly, the proposed formulation \eqref{ode_1}--\eqref{ode_2} provides a generalized portfolio optimization framework. It unifies several classical control settings as special cases, and offers a novel mathematical paradigm for modeling practical investment problems with performance-triggered uncertain exit times. 

\subsection{Time-optimal control under the deterministic setting}\label{sub-sec:deter}

The time-optimal control problem is an interesting yet challenging area in optimal control theory. The monograph by Pontryagin et. al. \cite{P62} first introduced the time-optimal control problem in the deterministic case. We begin by reviewing the time-optimal control problem in the deterministic setting. In this subsection, let $\sigma(\cdot)\equiv0,\ f(\cdot)\equiv1,\ \Psi(\cdot)\equiv0,\ g(\cdot)\equiv0$ and $h(\cdot)$ be a function of state $X^u(\cdot)$ and control $u(\cdot)$. Then the time-optimal control problem is formulated as follows:
\begin{equation}
J(u(\cdot))= \tau^{u}, \label{time-cost-1}
\end{equation}
where $\tau^u$ satisfies
$$
X^u(\tau^{u})\in \mathcal{D}=\{x:\ \Phi(x)\leq 0 \}, \label{free-time-1}
$$
and $\Phi(\cdot)\in C^{1}(\mathbb{R}^m)$. Thus, the objective of the time-optimal control problem is to determine the minimum time required to reach a specified domain $\mathcal{D}$. Given a fixed terminal time $T$, the time-optimal control problem is equivalent to minimizing the value
\begin{equation}
\label{time-cons-1}
\tau^u=\inf\bigg{\{}t:\Phi(X^u(t))\leq 0,\ t\in [0,T] \bigg{\}}\wedge T.
\end{equation}

Let $Y^u(t)=\Phi(X^u(t))$ and note that function $\Phi(\cdot)$ is differentiable in $x$. Then $Y^u(t)$ satisfies
\begin{equation}\label{time-ode_2}
\left\{
\begin{aligned}
&\text{d}{Y}^u(t)=h(X^u(t),u(t))\text{d}t, \quad t\in(0,T],\\
&Y^u(0)=y_0,
\end{aligned}
\right.
\end{equation}
where, $y_0=\Phi(x_0)$ and $h(x,u)=\Phi_x^{\top}(x)b(x,u)$. Thus, $\tau^u$ can be defined as
\begin{equation}
\label{time-cons-2}
\tau^u=\inf\bigg{\{}t:Y^u(t)\leq 0,\ t\in [0,T] \bigg{\}}\wedge T.
\end{equation}
Based on definition of $\tau^u$ (\ref{time-cons-2}), one may consider minimizing $\tau^u$ with a general integral function $h(\cdot)$ used to define the process $Y^u(\cdot)$. It is noteworthy that the classical time-optimal control problem is a special case of our framework.

Next, we provide a more detailed exposition of our time-optimal control framework.
\begin{remark}
The state $Y^u(\cdot)$ admits the following explicit representation
$$
Y^u(t)=y_0+\int_0^th(X^u(s),u(s))\text{d}s, \quad t\in(0,T].
$$
From the formulation of $Y^u(\cdot)$, the value of $Y^u(t)$ depends on the entire trajectory of $(X^u(\cdot),u(\cdot))$ over $[0,t]$. The monograph by Pontryagin et. al. \cite{P62} pioneered the study of time-optimal control in the deterministic case, which emerges as a special case in our more general setting. In \cite{P62}, $Y^u(\cdot)$ depends only on the current state $X^u(t)$. See \cite{E21} for further detailed on numerical methods for this time-optimal control problem.

Since $Y^u(\cdot)$ is continuous on $[0,T]$, if $\tau^u<T$, we have $Y^u(\tau^u)=0$, that is
$$
0=y_0+\int_0^{\tau^u}h(X^u(s),u(s))\text{d}s.
$$
This observation is key when investigating the related properties for this optimal control structure.
\end{remark}

\subsection{Time-optimal control under the stochastic setting}\label{subsec:time-cases}
In this part, we develop the time-optimal control problem within the stochastic optimal control structure. Let $f(\cdot)\equiv1,\ \Psi(\cdot)\equiv0$. Then the time-optimal control problem is formulated as follows:
\begin{equation}
J(u(\cdot))= \tau^{u}, \label{stoc-time-cost-1}
\end{equation}
where $\tau^u$ is defined by
\begin{equation}
\label{stoc-time-cons-1}
\tau^u=\inf\bigg{\{}t:\mathbb{E}[Y^u(t)]\leq 0,\ t\in [0,T] \bigg{\}}\wedge T,
\end{equation}
with
\begin{equation}\label{stoc-ode_1}
\left\{
\begin{aligned}
&\text{d}{Y}^u(t)=h(\mathbb{E}[X^u(t)],X^u(t),\mathbb{E}[b(X^u(t),u(t))],u(t))\text{d}t
+g(\mathbb{E}[X^u(t)],X^u(t),u(t))\text{d}W(t),\\
&Y^u(0)=y_0.
\end{aligned}
\right.
\end{equation}
Here, $h(\cdot)$ and $g(\cdot)$ are integral functions of $(\mathbb{E}[X^u(t)],X^u(t),u(t))$ and $\mathbb{E}[b(X^u(t),u(t))]$.
\begin{remark}
It is important to observe that $\tau^u$ represents a deterministic time parameter rather than a stopping time. In (\ref{stoc-time-cons-1}), $\tau^u$ depends on the expectation of process $Y^u(\cdot)$ that is the solution of a mean-field type SDE (\ref{stoc-ode_1}). Many practical problems align with this model. For example, one may seek the optimal strategy and  minimum time such that the return of a portfolio exceeds a given target. \cite{Y22} solved a varying terminal time mean-variance model with a constraint on the mean value of the portfolio asset, which moves with the varying terminal time. The results of \cite{Y22} suggested that for an investment plan requires minimizing the variance with a varying terminal time.
\end{remark}

\begin{example}\label{exm:cases}
$\tau^u$ is a deterministic functional of $u(\cdot)$. Below, we present some cases of process $Y^u(\cdot)$:

$(i)$. Let $h(\cdot)=b(\cdot),\ g(\cdot)=\sigma(\cdot)$ and $y_0=x_0$. Then
$$
Y^u(t)=X^u(t),\ 0\leq t\leq T
$$
and for $\tau^u<T$, we have
$$\mathbb{E}[X^u(\tau^u)]=0,\ \mathbb{E}[X^u(t)]<0,\ t<\tau^u.$$
Thus, $\tau^u$ is the minimum time such that mean value of the controlled state $X^u(\tau^u)$ reaches zero.

$(ii)$. Let $Y^u(t)=\Phi(\mathbb{E}[X^u(t)],X^u(t)),\ t\in [0,T]$, where $\Phi(\cdot)\in C^{1,2}(\mathbb{R}^m\times \mathbb{R}^m)$. Applying It\^{o} formula to $\Phi(\mathbb{E}[X^u(t)],X^u(t))$, we have that
\begin{equation}
\left\{
\begin{aligned}
\text{d}{Y}^u(t)=&\big[\partial_x\Phi(\mathbb{E}[X^u(t)],X^u(t))^{\top}\mathbb{E}[b^u(t)]
+\partial_y\Phi(\mathbb{E}[X^u(t)],X^u(t))^{\top}b^u(t)\\
&+\frac{1}{2}\sum_{j=1}^d{\sigma^{u,j}}^{\top}(t) \partial^2_{yy}\Phi(\mathbb{E}[X^u(t)],X^u(t))\sigma^{u,j}(t) \big]\text{d}t\\
&+\sum_{j=1}^d\partial_y\Phi(\mathbb{E}[X^u(t)],X^u(t))^{\top}\sigma^{u,j}(t)\text{d}W^j(t) ,\quad t\in(0,T],\\
Y^u(0)=&\Phi(x_0,x_0),
\end{aligned}
\right.
\end{equation}
where $b^u(t):=b(t,X^u(t),u(t)),\ \sigma^u(t):=\sigma(t,X^u(t),u(t)),\ \sigma^u(t)=(\sigma^{u,1},\sigma^{u,2},\cdots,\sigma^{u,d})$. Set
\begin{equation}
\left\{
\begin{aligned}
&h(\mathbb{E}[X^u(t)],X^u(t),\mathbb{E}[b^u(t)],u(t))=\partial_x\Phi(\mathbb{E}[X^u(t)],X^u(t))\mathbb{E}[b^u(t)]
+\partial_y\Phi(\mathbb{E}[X^u(t)],X^u(t))b^u(t)\\
&\qquad\qquad\qquad\qquad\qquad\qquad\qquad\quad +\frac{1}{2}\sum_{j=1}^d{\sigma^{u,j}}^{\top}(t) \partial^2_{yy}\Phi(\mathbb{E}[X^u(t)],X^u(t))\sigma^{u,j}(t),\\
&g(\mathbb{E}[X^u(t)],X^u(t),u(t))=\sum_{j=1}^d\partial_y\Phi(\mathbb{E}[X^u(t)],X^u(t))\sigma^{u,j}(t),\\
&y_0=\Phi(x_0,x_0).
\end{aligned}
\right.
\end{equation}
Thus, the case $Y^u(t)=\Phi(\mathbb{E}[X^u(t)],X^u(t))$ is a special case of equation (\ref{stoc-ode_1}).
\end{example}

\subsection{Traditional stochastic optimal control problem}\label{sub-sec:comp-stoch}

In the classical stochastic optimal control problem, one typically considers the optimal control theory with a fixed terminal time $T$-that is, the cost functional is
\begin{equation}
J(u(\cdot))=%
\mathbb{E}\bigg{[}{\displaystyle \int \limits_{0}^{T}}
f(X^u{(t)},u(t))\text{d}t+\Psi(X^u(T))\bigg{]}. \label{comp-cost-1}%
\end{equation}

In many practical problems, it often necessary to reach a target before the fixed terminal time $T$, where the target depends on the distribution of the controlled state process $X^u(\cdot)$.
For example, one may seek an optimal investment strategy $u(\cdot)$ that maximizes the expected utility of the varying terminal time wealth $X^u(\tau^u)$, subject to a variance constraint on wealth $X^u(\tau^u)$, where $\tau^u$ is a varying terminal time changing with the control $u(\cdot)$. The proposed model guarantees time-optimal attainment of the target.
Another objective is to design a control policy $u(\cdot)$ that minimizes the expected energy consumption $\mathbb{E}\left[\int_0^{\tau^u} |u(t)|^2 dt\right]$ while satisfying the precision constraint $\mathbb{E}[X^u(t)]\leq \alpha,\ t\in [0,\tau^u]$, where $\alpha$ is a given constant. These problems fall within our unified optimal control structure:
\begin{equation}
J(u(\cdot))=%
\mathbb{E}\bigg{[}{\displaystyle \int \limits_{0}^{\tau^{u}}}
f(X^u{(t)},u(t))\text{d}t+\Psi(X^u(\tau^u))\bigg{]},\label{comp-cost-2}%
\end{equation}
with a varying terminal time according to the constraint
\begin{equation}
\label{comp-time-1}
\tau^u=\inf\bigg{\{}t:\mathbb{E}[Y^u(t)]\leq 0,\ t\in [0,T] \bigg{\}}\wedge T.
\end{equation}

\section{Stochastic maximum principle}\label{sec:SMP}

In this section, we establish the local stochastic maximum principle for the control problem (\ref{cost-1})--(\ref{ode_2}). Given an optimal pair $(\bar{u}(\cdot),\bar{X}(\cdot))$, let $0<\rho<1$, and $v(\cdot)+\bar{u}(\cdot)\in \mathcal{U}[0,T]$. Define
\[
u^{\rho}(t)=\bar{u}(t)+\rho v(t)=(1-\rho)\bar{u}(t)+\rho(v(t)+\bar{u}(t)),\ t\in [0,T]
\]
which belongs to $\mathcal{U}[0,T]$. Denote by $X^{\rho}(\cdot)$ the solution to equation (\ref{ode_1}) under control $u^{\rho}(\cdot)$, and by $Y^{\rho}(\cdot)$ the solution to equation (\ref{ode_2}) driven by both control $u^{\rho}(\cdot)$ and state trajectory $X^{\rho}(\cdot)$.

To analyze the variation of the minimum time $\tau^{\bar{u}}$, we define $\tau^{u^{\rho}}$ corresponding to control $u^{\rho}(\cdot)\in \mathcal{U}[0,T]$. We first show that $\tau^{{u}_{\rho}}$ converges to $\tau^{\bar{u}}$ as $\rho\to 0$ in Lemma \ref{le-0}, and prove that $\tau^{u}$ is differentiable and continuous at $\bar{u}(\cdot)$ in Lemma \ref{le-2}. In the end, we show the proof of Theorem \ref{the-Max}.
\begin{lemma}
\label{le-0}
Let Assumptions \ref{ass-b}, \ref{assb-b2}, \ref{ass-fai} and \ref{ass-h} hold. Then, we have
\begin{equation}
\label{le0-e0}
\lim_{\rho\to 0}\left|{\tau^{\bar{u}}-\tau^{u^{\rho}}}\right|
=0.
\end{equation}
\end{lemma}
\begin{proof}
The proof follows arguments analogous to those in Lemma 2 of \cite{Y20}.
Thus, we omit the details for brevity.
\end{proof}

Let $y^{}(\cdot)$  be the solution
of the following variational equation:%
\begin{equation}
\label{apro-1}
\begin{array}
[c]{rl}%
\text{d}{y}(t)= & \big{[}b_{x}(\bar{X}{(t)},\bar{u}(t))y(t)+b_{u}(\bar{X}{(t)},\bar{u}(t)) v(t)\big{]}\text{d}t\\
&+ \displaystyle \sum_{j=1}^d\big{[} \sigma_{x}^j(\bar{X}{(t)},\bar{u}(t))y(t)+\sigma^j_u(\bar{X}{(t)},\bar{u}(t)) v(t)\big{]}\text{d}W^j(t), \\
y(0)= & 0,\quad t\in (0,T].
\end{array}
\end{equation}
The following lemma is classical, and we omit the proof. See \cite{B81} and \cite{B78} for further details.
\begin{lemma}
\label{le-1} Let Assumptions \ref{ass-b}, \ref{assb-b2} and \ref{ass-fai} hold. We have
\begin{equation}%
\begin{array}
[l]{l}%
\displaystyle\lim_{\rho\to 0}\displaystyle\sup_{t\in \lbrack0,T]}\mathbb{E}\left|\rho^{-1} (X^{\rho}(t)-\bar{X}(t))-y(t)\right| =0.  \\
\end{array}
 \label{var-1}%
\end{equation}
\end{lemma}

\begin{lemma}
\label{le-diff} Let Assumptions \ref{ass-b}, \ref{assb-b2} and \ref{ass-fai} hold. Thus, we have
\begin{equation}\label{equ-diff}
\begin{array}
[c]{rl}%
\displaystyle \lim_{\rho\to 0}\frac{G^{{u}^{\rho}}(t)-G^{\bar{u}}(t)}{\rho}=\bar{h}(t,v(t)),
\end{array}
\end{equation}
where
$$
\bar{h}(t,v(t))=\mathbb{E}\big[\bar{h}_{x_1}(t)^{\top}\mathbb{E}[y(t)]
+\bar{h}_{x_2}(t)^{\top}y(t)
+\bar{h}_{x_3}(t)^{\top}\frac{\mathrm{d}\mathbb{E}[y(t)]}{\mathrm{d}t}
+\bar{h}_{x_4}(t)^{\top}v(t)\big],
$$
$\bar{h}_{x_i}(t)$ denotes the partial derivative of $\bar{h}(t)$ with respect to the $i$-th variable, and
$$
\bar{h}(t)=h(\mathbb{E}[\bar{X}(t)],\bar{X}(t),\mathbb{E}[b(\bar{X}(t),\bar{u}(t))],\bar{u}(t)).
$$
\end{lemma}
\begin{proof}
From  the definition of function $G^u(\cdot)$, we have
\begin{equation*}%
\begin{array}
[c]{rl}%
&G^{{u}^{\rho}}(t)-G^{\bar{u}}(t)\\
=&\mathbb{E}[h(\mathbb{E}[X^{\rho}(t)],X^{\rho}(t),\mathbb{E}[b(X^{\rho}(t),u^{\rho}(t))],u^{\rho}(t)) -h(\mathbb{E}[\bar{X}(t)],\bar{X}(t),\mathbb{E}[b(\bar{X}(t),\bar{u}(t))],\bar{u}(t))]\\                               =&\mathbb{E}\big[\bar{h}_{x_1}(t)^{\top}\mathbb{E}[X^{\rho}(t)-\bar{X}(t)]
+\bar{h}_{x_2}(t)^{\top}[X^{\rho}(t)-\bar{X}(t)]
+\bar{h}_{x_3}(t)^{\top}\mathbb{E}[b(X^{\rho}(t),u^{\rho}(t))-b(\bar{X}(t),\bar{u}(t))]\\
&+\bar{h}_{x_4}(t)^{\top}[u^{\rho}(t)-\bar{u}(t)]\big]+o(\rho)\\
=&\mathbb{E}\big[\bar{h}_{x_1}(t)^{\top}\mathbb{E}[X^{\rho}(t)-\bar{X}(t)]
+\bar{h}_{x_2}(t)^{\top}[X^{\rho}(t)-\bar{X}(t)]
\displaystyle+\bar{h}_{x_3}(t)^{\top}\frac{\text{d}\mathbb{E}[y(t)]}{\text{d}t}\\
&+\bar{h}_{x_4}(t)^{\top}[u^{\rho}(t)-\bar{u}(t)]\big]+o(\rho),
\end{array}
\end{equation*}
where $\bar{h}(t)=h(\mathbb{E}[\bar{X}(t)],\bar{X}(t),\mathbb{E}[b(\bar{X}(t),\bar{u}(t))],\bar{u}(t))$.
Applying Lemma \ref{le-1}, we have
\begin{equation*}%
\begin{array}
[c]{rl}%
&\displaystyle \frac{G^{{u}^{\rho}}(t)-G^{\bar{u}}(t)}{\rho}\\                              =&\mathbb{E}\big[\bar{h}_{x_1}(t)^{\top}\mathbb{E}[y(t)]
+\bar{h}_{x_2}(t)^{\top}y(t)
\displaystyle+\bar{h}_{x_3}(t)^{\top}\frac{\text{d}\mathbb{E}[y(t)]}{\text{d}t}
+\bar{h}_{x_4}(t)^{\top}v(t)\big]+o(1),\\
\end{array}
\end{equation*}
which yields Lemma \ref{le-diff}.
\end{proof}

Next, we derive the dual representation of
$$
\displaystyle \int_0^{\tau^{\bar{u}}}\bar{h}(t,v(t))\mathrm{d}t
$$
which is useful in proving the maximum principle. Define
\begin{equation*}
K(x,x',b(x,u),u,p_0,q_0)=b(x,u)^{\top}p_0+\sum_{j=1}^d\sigma^j(x,u)^{\top}q_0^j-h(x',x,b(x,u),u),
\end{equation*}
where $(x,x',u,p_0,q_0)\in \mathbb{R}^m\times \mathbb{R}^m\times U\times \mathbb{R}^m\times \mathbb{R}^{m\times d}$,
and consider the first-order adjoint equation:
\begin{equation}%
\begin{array}
[c]{rl}%
-\text{d}{p}_0(t)= & \bigg{[}b_x(\bar{X}(t),\bar{u}(t))^{\top}p_0(t)+ \displaystyle\sum_{j=1}^d\sigma_x^{j}(\bar{X}(t),\bar{u}(t))^{\top}q_0^j(t)
                   -\bar{h}_{x_2}(t)\bigg{]}\text{d}t \\
                   &-\bigg{[}\mathbb{E}[\bar{h}_{x_1}(t)]+b_x(\bar{X}(t),\bar{u}(t))^{\top}\mathbb{E}[\bar{h}_{x_3}(t)] \bigg{]}\text{d}t \\
                   &-q_0(t)\text{d}W(t),\ t\in[0,\tau^{\bar{u}}),\\
p_0(\tau^{\bar{u}})= &0.
\end{array}
\label{line-prin-1}%
\end{equation}
For notational simplicity, $\bar{h}(t)=h(\mathbb{E}[\bar{X}(t)],\bar{X}(t),\mathbb{E}[b(\bar{X}(t),\bar{u}(t))],\bar{u}(t))$, and $\bar{h}_{x_i}(t)$ denotes the partial derivative with respect to the $i$-th variable of $h(\cdot)$, $i=1,2,3,4$.

\begin{lemma}
\label{le-dual} Let Assumptions \ref{ass-b}, \ref{assb-b2} and \ref{ass-fai} hold. We have
\begin{equation}\label{equ-dual}
\begin{array}
[c]{rl}%
\displaystyle \int_0^{\tau^{\bar{u}}}\bar{h}(t,v(t))\mathrm{d}t=
\displaystyle-\mathbb{E}\int_0^{\tau^{\bar{u}}}
\big{[}\hat{K}(t)v(t)\big{]}\text{d}t,
\end{array}
\end{equation}
where
\begin{equation*}%
\begin{array}
[c]{rl}%
\hat{K}(t)=&{K}_u(\bar{X}{(t)},\mathbb{E}[\bar{X}{(t)}],\mathbb{E}[b(\bar{X}{(t)},\bar{u}(t))],\bar{u}(t),
p_0(t),q_0(t))\\
&+\bar{h}_{x_3}(t)^{\top}
\mathbb{E}[b_u(\bar{X}(t),\bar{u}(t))]-\mathbb{E}[\bar{h}_{x_3}(t)^{\top}]
b_u(\bar{X}(t),\bar{u}(t)),
\end{array}
\end{equation*}
${K}_u(\cdot)$ denotes the derivative with respect to $u$ and $\bar{h}(t,v(t))$ is defined in Lemma \ref{le-diff}.
\end{lemma}
\begin{proof}
Applying It\^{o} formula to $p_0(t)^{\top}y(t)$, we obtain
\begin{equation*}%
\begin{array}
[c]{rl}%
\displaystyle \text{d}(p_0(t)^{\top}y(t))=&\text{d}p_0(t)^{\top}y(t)+p_0(t)^{\top}\text{d}y(t)+\text{d}p_0(t)^{\top}\text{d}y(t)\\                              =&-\bigg{[}p_0(t)^{\top}b_x(\bar{X}(t),\bar{u}(t))+ \displaystyle\sum_{j=1}^dq_0^j(t)^{\top}\sigma_x^{j}(\bar{X}(t),\bar{u}(t))-\bar{h}_{x_2}(t)^{\top}\bigg{]}y(t)\text{d}t\\
&+\bigg{[}\mathbb{E}[\bar{h}_{x_1}(t)^{\top}]+\mathbb{E}[\bar{h}_{x_3}(t)^{\top}]b_x(\bar{X}(t),\bar{u}(t)) \bigg{]}y(t)\text{d}t\\
&+p_0(t)^{\top}\bigg{[}b_{x}(\bar{X}{(t)},\bar{u}(t))y(t)+b_{u}(\bar{X}{(t)},\bar{u}(t)) v(t) \bigg{]}\text{d}t\\
&\displaystyle +\sum_{j=1}^dq^j_0(t)^{\top}\big{[} \sigma_{x}^j(\bar{X}{(t)},\bar{u}(t))y(t)+\sigma^j_u(\bar{X}{(t)},\bar{u}(t)) v(t)\big{]}\text{d}t\\
&+M(t)\text{d}W(t),
\end{array}
\end{equation*}
where $M(t)$ is the martingale part. Since $y(0)=p_0(\tau^{\bar{u}})=0$, integrating on both sides from $0$ to $\tau^{\bar{u}}$ and taking expectation $\mathbb{E}[\cdot]$ yields
\begin{equation*}%
\begin{array}
[c]{rl}%
 0=&\displaystyle \mathbb{E}\int_0^{\tau^{\bar{u}}}\bigg{[}\bar{h}_{x_2}(t)^{\top}y(t)
+\mathbb{E}[\bar{h}_{x_1}(t)^{\top}]y(t)+\mathbb{E}[\bar{h}_{x_3}(t)^{\top}]b_x(\bar{X}(t),\bar{u}(t))y(t)   \bigg{]}\text{d}t\\
&+\displaystyle \mathbb{E}\int_0^{\tau^{\bar{u}}}\bigg{[}p_0(t)^{\top}b_{u}(\bar{X}{(t)},\bar{u}(t)) v(t)  +\sum_{j=1}^dq^j_0(t)^{\top}\sigma^j_u(\bar{X}{(t)},\bar{u}(t)) v(t)\bigg{]}\text{d}t.
\end{array}
\end{equation*}
Note that
$$
\bar{h}(t,v(t))=\mathbb{E}\big[\bar{h}_{x_1}(t)^{\top}\mathbb{E}[y(t)]
+\bar{h}_{x_2}(t)^{\top}y(t)
+\bar{h}_{x_3}(t)^{\top}\frac{\text{d}\mathbb{E}[y(t)]}{\text{d}t}
+\bar{h}_{x_4}(t)^{\top}v(t)\big],
$$
so it follows that
\begin{equation*}%
\begin{array}
[c]{rl}%
\displaystyle \int_0^{\tau^{\bar{u}}}\bar{h}(t,v(t))\text{d}t=&
\displaystyle -\mathbb{E}\int_0^{\tau^{\bar{u}}}\bigg{[}p_0(t)^{\top}b_{u}(\bar{X}{(t)},\bar{u}(t)) v(t)+\sum_{j=1}^dq^j_0(t)^{\top}\sigma^j_u(\bar{X}{(t)},\bar{u}(t)) v(t)\bigg{]}\text{d}t\\
&\displaystyle+\mathbb{E}\int_0^{\tau^{\bar{u}}}\bigg{[}\mathbb{E}[\bar{h}_{x_3}(t)^{\top}]
b_u(\bar{X}(t),\bar{u}(t))v(t)+\bar{h}_{x_4}(t)^{\top}v(t)\bigg{]}\text{d}t\\
=&\displaystyle-\mathbb{E}\int_0^{\tau^{\bar{u}}}
\big{[}K_u(\bar{X}{(t)},\mathbb{E}[\bar{X}{(t)}],\mathbb{E}[b(\bar{X}{(t)},\bar{u}(t))],
\bar{u}(t),p_0(t),q_0(t))v(t)\big{]}\text{d}t\\
&-\displaystyle\mathbb{E}\int_0^{\tau^{\bar{u}}}
\big{[}\bar{h}_{x_3}(t)^{\top}
\mathbb{E}[b_u(\bar{X}(t),\bar{u}(t))]-\mathbb{E}[\bar{h}_{x_3}(t)^{\top}]
b_u(\bar{X}(t),\bar{u}(t))\big{]}v(t)\text{d}t\\
=&\displaystyle-\mathbb{E}\int_0^{\tau^{\bar{u}}}
\big{[}\hat{K}(t)v(t)\big{]}\text{d}t,
\end{array}
\end{equation*}
where ${K}_u(\cdot)$ denotes the the derivative with respect to $u$. This completes the proof.
\end{proof}

The proof of the main results is based on a case by case analysis, divided as follows.
\begin{lemma}
\label{le-2}
Let Assumptions \ref{ass-b}, \ref{assb-b2}, \ref{ass-fai} and \ref{ass-h} hold.
Then, we have the following results:

 (i). If $\tau^{\bar{u}}<T$, then
\begin{equation}
\label{le2-e0}
\lim_{\rho\to 0}\frac{\tau^{\bar{u}}-\tau^{u^{\rho}}}{\rho}
=\int_0^{\tau^{\bar{u}}}\frac{ \bar{h}(t,v(t))}{G^{\bar{u}}(\tau^{\bar{u}})}\mathrm{d}t.
\end{equation}

(ii). If $\bigg{\{}t:\mathbb{E}[Y^{\bar{u}}(t)]\leq 0,\ t\in [0,T] \bigg{\}}=\varnothing$, we have
\begin{equation}
\lim_{\rho\to 0}\frac{\tau^{\bar{u}}-\tau^{u^{\rho}}}{\rho}=0.
\end{equation}

(iii). If $\inf\bigg{\{}t:\mathbb{E}[Y^{\bar{u}}(t)]\leq 0,\ t\in [0,T] \bigg{\}}=T$, then there exists a sequence $\rho_n\to 0$ as $n\to +\infty$ such that
\begin{equation}
\lim_{n\to +\infty}\frac{\tau^{\bar{u}}-\tau^{u^{\rho_n}}}{\rho_n}=
\int_0^{\tau^{\bar{u}}}\frac{ \bar{h}(t,v(t))}{G^{\bar{u}}(\tau^{\bar{u}})}\mathrm{d}t\ \  \mathrm{or}\ \  0.
\end{equation}
\end{lemma}
\begin{proof}
We first consider case (i). From equation (\ref{mean-1}), we have
$$
\mathbb{E}[Y^{\bar{u}}(s)]=y_0+\int_0^sG^{\bar{u}}(t)\text{d}t,
$$
and
$$
\mathbb{E}[Y^{{u}^{\rho}}(s)]=y_0+\int_0^sG^{{u}^{\rho}}(t)\text{d}t.
$$
For sufficiently small $\rho\neq 0$, it follows that
$$
0=\mathbb{E}[Y^{\bar{u}}(\tau^{\bar{u}})]=y_0+\int_0^{\tau^{\bar{u}}}G^{\bar{u}}(t)\text{d}t,
$$
and
$$
0=\mathbb{E}[Y^{{u}^{\rho}}(\tau^{u^{\rho}})]=y_0+\int_0^{\tau^{u^{\rho}}}G^{{u}^{\rho}}(t)\text{d}t.
$$
Thus
$$
0=\int_0^{\tau^{\bar{u}}}G^{\bar{u}}(t)\text{d}t-\int_0^{\tau^{u^{\rho}}}G^{{u}^{\rho}}(t)\text{d}t,
$$
which implies
\begin{equation}\label{equ-var}
\int_0^{\tau^{\bar{u}}}[G^{\bar{u}}(t)-G^{{u}^{\rho}}(t)]\text{d}t=
\int_{\tau^{\bar{u}}}^{\tau^{u^{\rho}}}G^{{u}^{\rho}}(t)\text{d}t.
\end{equation}

By Lemma \ref{le-0} and \ref{le-diff}, equation (\ref{equ-var}) can be rewritten as
\begin{equation*}%
\begin{array}
[c]{rl}%
&\displaystyle \int_0^{\tau^{\bar{u}}}\frac{G^{\bar{u}}(t)-G^{{u}^{\rho}}(t)}{\rho}\text{d}t
=\displaystyle \frac{\tau^{u^{\rho}}-{\tau^{\bar{u}}}}{\rho}[G^{\bar{u}}(\tau^{\bar{u}})+o(1)]
\end{array}
\end{equation*}
and hence
$$
\displaystyle \lim_{\rho\to 0}\frac{{\tau^{\bar{u}}}-\tau^{u^{\rho}}}{\rho}
=\int_0^{\tau^{\bar{u}}}\frac{\bar{h}(t,v(t))}{G^{\bar{u}}(\tau^{\bar{u}})}\text{d}t.
$$

Now consider case (ii). Here, $\bigg{\{}t:\mathbb{E}[Y^{\bar{u}}(t)]\leq 0,\ t\in [0,T] \bigg{\}}=\varnothing$, so $\mathbb{E}[Y^{\bar{u}}(t))]>0,\ t\in [0,T]$. Since
$$
\mathbb{E}[Y^u(s)]=y_0+\int_0^sG^{{u}}(t)\mathrm{d}t,
$$
and
$$
G^{{u}}(t)=\mathbb{E}[h(\mathbb{E}[X^u(t)],X^u(t),\mathbb{E}[b(X^u(t),u(t))],u(t))],\ t\in [0,T],
$$
for sufficiently small $\rho\neq 0$, we have $\mathbb{E}[Y^{{u}^{\rho}}(t))]>0$, implying $\bigg{\{}t:\mathbb{E}[Y^{{u}^{\rho}}(t)]\leq 0,\ t\in [0,T] \bigg{\}}=\varnothing$ and $\tau^{u^{\rho}}=T$. Thus
$$
\lim_{\rho\to 0}\frac{\tau^{\bar{u}}-\tau^{u^{\rho}}}{\rho}=0.
$$

Finally, consider case (iii). The condition $\inf\bigg{\{}t:\mathbb{E}[Y^{\bar{u}}(t)]\leq 0,\ t\in [0,T] \bigg{\}}=T$ means that $\mathbb{E}[Y^{\bar{u}}(\tau^{\bar{u}})]=0$ and $\tau^{\bar{u}}=T$. For sufficiently small $\rho\neq 0$, either $\tau^{u^{\rho}}<T$ or $\bigg{\{}t:\mathbb{E}[Y^{{u}^{\rho}}(t)]\leq 0,\ t\in [0,T] \bigg{\}}=\varnothing$. Therefore, case (iii) combines the results of cases (i) and (ii), i.e., there exists a sequence $\rho_n\to 0$ as $n\to +\infty$ such that
\begin{equation}
\lim_{n\to +\infty}\frac{\tau^{\bar{u}}-\tau^{u^{\rho_n}}}{\rho_n}=
\int_0^{\tau^{\bar{u}}}\frac{ \bar{h}(t,v(t))}{G^{\bar{u}}(\tau^{\bar{u}})}\mathrm{d}t\ \  \mathrm{or}\ \  0.
\end{equation}
\end{proof}

We now derive the variational equation for cost functional (\ref{cost-1}) in the following lemma.
\begin{lemma}
\label{le-3} Let Assumptions \ref{ass-b}, \ref{assb-b2}, \ref{ass-fai} and \ref{ass-h} hold. Then, we have the following results.

 (i). If $\tau^{\bar{u}}<T$, then
\begin{equation}
\label{le3-e0}
\begin{array}
[c]{rl}
& \rho^{-1}\left[J(u^{\rho}(\cdot))-J(\bar{u}(\cdot))\right] \\
= &\displaystyle -\int_0^{\tau^{\bar{u}}}\frac{ \mathbb{E}[\tilde{\Psi}^{\bar{u}}(\tau^{\bar{u}})+f(\bar{X}(\tau^{\bar{u}}),\bar{u}(\tau^{\bar{u}}))]
\bar{h}(t,v(t))}{G^{\bar{u}}(\tau^{\bar{u}})}\mathrm{d}t
+\mathbb{E}\bigg{[} \Psi_x(\bar{X}(\tau^{\bar{u}}))^{\top}y(\tau^{\bar{u}})\big{)}\bigg{]}\\
&+\displaystyle \mathbb{E}\int_{0}^{\tau^{\bar{u}}}\bigg{[}
f_x(\bar{X}{(t)},\bar{u}(t))^{\top}y(t)+f_u(\bar{X}{(t)},\bar{u}(t))^{\top}v(t)  \bigg{]}\mathrm{d}t+o(1),
\end{array}
\end{equation}
where
$$
\tilde{\Psi}^{\bar{u}}(\tau^{\bar{u}})=\Psi_x(\bar{X}(\tau^{\bar{u}}))^{\top}b(\bar{X}(\tau^{\bar{u}}),\bar{u}(\tau^{\bar{u}}))
+\frac{1}{2}\sum_{j=1}^d\sigma^j(\bar{X}(\tau^{\bar{u}}),\bar{u}(\tau^{\bar{u}}))^{\top}\Psi_{xx}(\bar{X}(\tau^{\bar{u}}))
\sigma^j(\bar{X}(\tau^{\bar{u}}),\bar{u}(\tau^{\bar{u}})).
$$

(ii). If $\bigg{\{}t:\mathbb{E}[Y^{\bar{u}}(t))]\leq 0,\ t\in [0,T] \bigg{\}}=\varnothing$, then
\begin{equation}
\begin{array}
[c]{rl}
& \rho^{-1}\left[J(u^{\rho}(\cdot))-J(\bar{u}(\cdot))\right] \\
= &\mathbb{E}\bigg{[} \Psi_x(\bar{X}(\tau^{\bar{u}}))^{\top}y(\tau^{\bar{u}})\big{)}\bigg{]}
+\displaystyle\int_{0}^{\tau^{\bar{u}}}\mathbb{E}\bigg{[}
f_x(\bar{X}{(t)},\bar{u}(t))^{\top}y(t)+f_u(\bar{X}{(t)},\bar{u}(t))^{\top}v(t)  \bigg{]}\mathrm{d}t+o(1).
\end{array}
\end{equation}

(iii). If $\inf\bigg{\{}t:\mathbb{E}[Y^{\bar{u}}(t))]\leq 0,\ t\in [0,T] \bigg{\}}=T$, then
\begin{equation}\label{equa-iii-1}
\begin{array}
[c]{rl}
& \rho^{-1}\left[J(u^{\rho}(\cdot))-J(\bar{u}(\cdot))\right] \\
= &\displaystyle -\int_0^{\tau^{\bar{u}}}\frac{ \mathbb{E}[\tilde{\Psi}^{\bar{u}}(\tau^{\bar{u}})+f(\bar{X}(\tau^{\bar{u}}),\bar{u}(\tau^{\bar{u}}))]
\bar{h}(t,v(t))}{G^{\bar{u}}(\tau^{\bar{u}})}\mathrm{d}t
+\mathbb{E}\bigg{[} \Psi_x(\bar{X}(\tau^{\bar{u}}))^{\top}y(\tau^{\bar{u}})\big{)}\bigg{]}\\
&+\displaystyle\int_{0}^{\tau^{\bar{u}}}\mathbb{E}\bigg{[}
f_x(\bar{X}{(t)},\bar{u}(t))^{\top}y(t)+f_u(\bar{X}{(t)},\bar{u}(t))^{\top}v(t)  \bigg{]}\mathrm{d}t+o(1),
\end{array}
\end{equation}
or
\begin{equation}\label{equa-iii-2}
\begin{array}
[c]{rl}
& \rho^{-1}\left[J(u^{\rho}(\cdot))-J(\bar{u}(\cdot))\right] \\
= &\displaystyle \mathbb{E}\bigg{[} \Psi_x(\bar{X}(\tau^{\bar{u}}))^{\top}y(\tau^{\bar{u}})\big{)}\bigg{]}+\displaystyle\int_{0}^{\tau^{\bar{u}}}\mathbb{E}\bigg{[}
f_x(\bar{X}{(t)},\bar{u}(t))^{\top}y(t)+f_u(\bar{X}{(t)},\bar{u}(t))^{\top}v(t)  \bigg{]}\mathrm{d}t+o(1).
\end{array}
\end{equation}
\end{lemma}
\begin{proof}
First, consider case (i). Since $\tau^{\bar{u}}<T$, for sufficiently small $\rho\neq 0$, we have $\tau^{u^{\rho}}<T$, and
\begin{equation}\label{equa-(i)}
\begin{array}
[c]{rl}
& J(u^{\rho}(\cdot))-J(\bar{u}(\cdot)) \\
= &\displaystyle \mathbb{E}\bigg{[}{ \int \limits_{0}^{\tau^{u^{\rho}}}} f(X^{u^{\rho}}{(t)},u^{\rho}(t))\text{d}t+\Psi(X^{u^{\rho}}(\tau^{u^{\rho}}))\bigg{]}
-\mathbb{E}\bigg{[}{ \int \limits_{0}^{\tau^{\bar{u}}}} f(\bar{X}{(t)},\bar{u}(t))\text{d}t+\Psi(\bar{X}(\tau^{\bar{u}}))\bigg{]}\\
:=& \displaystyle I_1+I_2+I_3,
\end{array}
\end{equation}
where
\begin{equation*}
\begin{array}
[c]{rl}
&\displaystyle I_1= \mathbb{E}\bigg{[}{ \int \limits_{0}^{\tau^{u^{\rho}}}} f(X^{u^{\rho}}{(t)},u^{\rho}(t))\text{d}t- \int \limits_{0}^{\tau^{\bar{u}}} f(\bar{X}{(t)},\bar{u}(t))\text{d}t\bigg{]},\\
&I_2=\mathbb{E}\bigg{[}\Psi(X^{u^{\rho}}(\tau^{u^{\rho}}))-\Psi(X^{u^{\rho}}(\tau^{\bar{u}}))\bigg{]},\\
&I_3=\mathbb{E}\bigg{[}\Psi(X^{u^{\rho}}(\tau^{\bar{u}}))-\Psi(\bar{X}(\tau^{\bar{u}}))\bigg{]}.
\end{array}
\end{equation*}
Applying It\^{o} formula to $\Psi(\cdot)$, we rewrite term $I_2$ as
$$
I_2=\mathbb{E}\bigg{[}{\int\limits_{\tau^{\bar{u}}}^{\tau^{u^{\rho}}}} \tilde{\Psi}^{{u}^{\rho}}(t)\text{d}t\bigg{]},
$$
where the function $\tilde{\Psi}^{{u}}(t)$ is defined by
$$
\tilde{\Psi}^{{u}}(t):=\Psi_x(X^u(t))^{\top}b({X}^{u}(t),{u}(t))
+\frac{1}{2}\sum_{j=1}^d\sigma^j(X^u(t),u(t))^{\top}\Psi_{xx}({X}^u(t))
\sigma^j({X}^u(t),{u}(t)).
$$
Thus, $I_1$ and $I_2$ can be handled similarly. Now, compute $I_1$:
\begin{equation*}
\begin{array}
[c]{rl}
I_1= &\displaystyle \mathbb{E}\bigg{[}{ \int \limits_{0}^{\tau^{u^{\rho}}}} f(X^{u^{\rho}}{(t)},u^{\rho}(t))\text{d}t- \int \limits_{0}^{\tau^{\bar{u}}} f(\bar{X}{(t)},\bar{u}(t))\text{d}t\bigg{]}\\
=&\displaystyle \mathbb{E}\bigg{[}{\int \limits_{\tau^{\bar{u}}}^{\tau^{u^{\rho}}}} f(X^{u^{\rho}}{(t)},u^{\rho}(t))\text{d}t+\int\limits_{0}^{\tau^{\bar{u}}} [f(X^{u^{\rho}}{(t)},u^{\rho}(t))-f(\bar{X}{(t)},\bar{u}(t))]\text{d}t\bigg{]}\\
=&\displaystyle -\int_0^{\tau^{\bar{u}}}\frac{\rho\mathbb{E}[f(\bar{X}(\tau^{\bar{u}}),\bar{u}(\tau^{\bar{u}}))]
\bar{h}(t,v(t))}{G^{\bar{u}}(\tau^{\bar{u}})}\mathrm{d}t\\
&\displaystyle+\rho\mathbb{E}\int_{0}^{\tau^{\bar{u}}}\bigg{[}
f_x(\bar{X}{(t)},\bar{u}(t))^{\top}y(t)+f_u(\bar{X}{(t)},\bar{u}(t))^{\top}v(t)  \bigg{]}\mathrm{d}t
+o(\rho),
\end{array}
\end{equation*}
where the third equality follows from (i) of Lemma \ref{le-2} and Lemma \ref{le-1}.

Similar to the proof in term $I_1$, we have the following results for $I_2$ and $I_3$
$$
I_2=\displaystyle -\int_0^{\tau^{\bar{u}}}\frac{ \rho\mathbb{E}[\tilde{\Psi}^{\bar{u}}(\tau^{\bar{u}})]
\bar{h}(t,v(t))}{G^{\bar{u}}(\tau^{\bar{u}})}\mathrm{d}t+o(\rho)
$$
and
$$
I_3=\rho \mathbb{E}\bigg{[} \Psi_x(\bar{X}(\tau^{\bar{u}}))^{\top}y(\tau^{\bar{u}})\big{)}\bigg{]}+o(\rho).
$$
Combining these representations of $I_1,I_2,I_3$ with equation (\ref{equa-(i)}) yields (\ref{le3-e0}).

Now consider case (ii). Since $\bigg{\{}t:\mathbb{E}[Y^{\bar{u}}(t))]\leq 0,\ t\in [0,T] \bigg{\}}=\varnothing$, for sufficiently small $\rho\neq 0$, we have
$\bigg{\{}t:\mathbb{E}[Y^{{u}^{\rho}}(t))]\leq 0,\ t\in [0,T] \bigg{\}}=\varnothing$, so $\tau^{\bar{u}}=\tau^{u^{\rho}}=T$. This reduces to the classical case, and details are omitted.

For case (iii), the condition $\inf\bigg{\{}t:\mathbb{E}[Y^{\bar{u}}(t)]\leq 0,\ t\in [0,T] \bigg{\}}=T$ implies $\mathbb{E}[Y^{\bar{u}}(\tau^{\bar{u}})]=0$ and $\tau^{\bar{u}}=T$. For sufficiently small $\rho\neq 0$, either $\tau^{u^{\rho}}<T$ or $\bigg{\{}t:\mathbb{E}[Y^{{u}^{\rho}}(t)]\leq 0,\ t\in [0,T] \bigg{\}}=\varnothing$. Thus, case (iii) combines the results of cases (i) and (ii), i.e., there exists a sequence $\rho_n\to 0$ as $n\to +\infty$ such that equation (\ref{equa-iii-1}) or (\ref{equa-iii-2}) holds.
\end{proof}

We introduce the following first-order adjoint equation:
\begin{equation}%
\begin{array}
[c]{rl}%
-\text{d}{p}(t)= & \bigg{[}b_x(\bar{X}{(t)},\bar{u}(t))^{\top}p(t)+ \displaystyle\sum_{j=1}^d\sigma_x^j(\bar{X}{(t)},\bar{u}(t))^{\top}q^j(t) \\
               &-f_x(\bar{X}{(t)},\bar{u}(t))\bigg{]}\text{d}t-q(t)\text{d}W(t),\ t\in[0,\tau^{\bar{u}}),\\
p(\tau^{\bar{u}})= &-\Psi_{x}(\bar{X}(\tau^{\bar{u}})).
\end{array}
\label{prin-1}%
\end{equation}
Equation (\ref{prin-1}) is a linear Backward stochastic differential equation, and its explicit solution can be obtained via the dual method, see Chapter 7 in \cite{Y99} for the basic theory of Backward stochastic differential equation. Define
\begin{equation*}
{H}(x,u,p,q)=b(x,u)^{\top}p+\sum_{j=1}^d\sigma^j (x,u)^{\top}q^j-f(x,u),\text{ \  \ }%
(x,u,p,q)\in \mathbb{R}^m\times U\times \mathbb{R}^m\times \mathbb{R}^{m\times d}.%
\end{equation*}

In the following, we show the proof of Theorem \ref{the-Max}.

\begin{proof} First, prove case (i). Since $\tau^{\bar{u}}<T$, for sufficiently small $\rho\neq 0$, we have $\tau^{u^{\rho}}<T$. From
\begin{equation}
J(\bar{u}(\cdot))= \underset{u(\cdot)\in\mathcal{U}[0,\tau^u]}{\inf}J(u(\cdot)),
\end{equation}
it follows that $\rho^{-1}\left[J(u^{\rho}(\cdot))-J(\bar{u}(\cdot))\right]\geq 0$. By (i) of Lemma \ref{le-3}, we have
\begin{equation}\label{equa-smp-1}
\begin{array}
[c]{rl}
0\leq &\displaystyle -\int_0^{\tau^{\bar{u}}}\frac{ \mathbb{E}[\tilde{\Psi}^{\bar{u}}(\tau^{\bar{u}})+f(\bar{X}(\tau^{\bar{u}}),\bar{u}(\tau^{\bar{u}}))]
\bar{h}(t,v(t))}{G^{\bar{u}}(\tau^{\bar{u}})}\mathrm{d}t
+\mathbb{E}\bigg{[}\Psi_x(\bar{X}(\tau^{\bar{u}}))^{\top}y(\tau^{\bar{u}})\big{)}\bigg{]}\\
&+\displaystyle \mathbb{E}\int_{0}^{\tau^{\bar{u}}}\bigg{[}
f_x(\bar{X}{(t)},\bar{u}(t))^{\top}y(t)+f_u(\bar{X}{(t)},\bar{u}(t))^{\top}v(t)  \bigg{]}\mathrm{d}t.\\
\end{array}
\end{equation}

Applying It\^{o} formula to $p(t)^{\top}y(t)$, we obtain
\begin{equation*}
\begin{array}
[c]{rl}
&\mathbb{E}\bigg{[}[p(\tau^{\bar{u}})^{\top}y(\tau^{\bar{u}})\bigg{]}\\
=&-\mathbb{E}\bigg{[}\Psi_x(\bar{X}(\tau^{\bar{u}}))^{\top}y(\tau^{\bar{u}})\big{)}\bigg{]}\\
=&\displaystyle \mathbb{E}\int_{0}^{\tau^{\bar{u}}}\bigg{[}H_u(\bar{X}(t),\bar{u}(t),p(t),q(t))v(t)
+f_x(\bar{X}{(t)},\bar{u}(t))^{\top}y(t)+f_u(\bar{X}{(t)},\bar{u}(t))^{\top}v(t)  \bigg{]}\mathrm{d}t.
\end{array}
\end{equation*}
Combining with equation (\ref{equa-smp-1}), we get
$$
\displaystyle 0\leq -\mathbb{E}\int_{0}^{\tau^{\bar{u}}}\bigg{[}H_u(\bar{X}(t),\bar{u}(t),p(t),q(t))v(t)
+ \frac{ \mathbb{E}[\tilde{\Psi}^{\bar{u}}(\tau^{\bar{u}})+f(\bar{X}(\tau^{\bar{u}}),\bar{u}(\tau^{\bar{u}}))]
\bar{h}(t,v(t))}{G^{\bar{u}}(\tau^{\bar{u}})}\bigg{]}\mathrm{d}t.
$$
Let $v(t)=u-\bar{u}(t)$, by Lemma \ref{le-dual}, we have
$$
H_u(\bar{X}(t),\bar{u}(t),p(t),q(t))(u-\bar{u}(t))-\displaystyle\frac{ \mathbb{E}[\tilde{\Psi}^{\bar{u}}(\tau^{\bar{u}})
+f(\bar{X}(\tau^{\bar{u}}),\bar{u}(\tau^{\bar{u}}))]\hat{K}(t)(u-\bar{u}(t))}{G^{\bar{u}}
(\tau^{\bar{u}})}\leq 0.
$$

Next consider case (ii). Since $\bigg{\{}t:\mathbb{E}[Y^{\bar{u}}(t))]\leq 0,\ t\in [0,T] \bigg{\}}=\varnothing$, for sufficiently small $\rho\neq 0$, we have
$\bigg{\{}t:\mathbb{E}[Y^{{u}^{\rho}}(t))]\leq 0,\ t\in [0,T] \bigg{\}}=\varnothing$, so $\tau^{\bar{u}}=\tau^{u^{\rho}}=T$. This reduces to the classical case, and details are omitted.

Finally, consider case (iii). The condition $\inf\bigg{\{}t:\mathbb{E}[Y^{\bar{u}}(t)]\leq 0,\ t\in [0,T] \bigg{\}}=T$ implies $\mathbb{E}[Y^{\bar{u}}(\tau^{\bar{u}})]=0$ and $\tau^{\bar{u}}=T$. For sufficiently small $\rho\neq 0$, either $\tau^{u^{\rho}}<T$ or $\bigg{\{}t:\mathbb{E}[Y^{{u}^{\rho}}(t)]\leq 0,\ t\in [0,T] \bigg{\}}=\varnothing$. Thus, case (iii) combines the results of cases (i) and (ii), i.e., there exists a sequence $\rho_n\to 0$ as $n\to +\infty$ such that equation (\ref{equa-smp-2}) or (\ref{equa-smp-3}) holds.
\end{proof}

\section{Existence of Optimal Controls for Linear Systems}\label{sec:existence}

In this section, we show the existence of an optimal control for the minimum-time linear stochastic optimal control problem under mean constraints. While the previous sections derived necessary conditions (via the stochastic maximum principle), the existence of an optimal pair $(\bar{u}(\cdot),\bar{X}(\cdot))$ is a prerequisite for the validity of such conditions.
The proof follows the standard stochastic optimal control existence theory for linear systems (see \cite{Y99}).
We consider  the state $X^u(\cdot)$ satisfying a linear controlled stochastic differential equation
\begin{equation}\label{e-line-ode_1}
\left\{
\begin{aligned}
&\text{d}{X}^u(t)=[AX^u(t)+Bu(t)]\text{d}t+\sum_{j=1}^d[C_jX^u(t)+D_ju(t)]\text{d}W_j(t),\ t\in(0,T],\\
&X(0)=x_0,
\end{aligned}
\right.
\end{equation}
where $A,B,C_j,D_j,\ 1\leq j\leq d$ are constant coefficient matrices. $Y^u(t)$ satisfies
\begin{equation}\label{e-line-y-ode}
\left\{
\begin{aligned}
\text{d}{Y}^u(t)=&\big{[}E^{\top}_1\mathbb{E}[X^u(t)]+E^{\top}_2X^u(t)
+E^{\top}_3\mathbb{E}[AX^u(t)+Bu(t)]+E^{\top}_4u(t)\big{]}\text{d}t\\
&+g(\mathbb{E}[X^u(t)],X^u(t),u(t))\text{d}W(t)\\
Y^u(0)=&y_0,
\end{aligned}
\right.
\end{equation}
where $E^{\top}_1,E^{\top}_2,E^{\top}_3,E^{\top}_4$ are constant vectors. The basic assumption of this section is given as follows.
\begin{assumption}\label{ass-f-pai}
$U$ is a convex and bounded subset of $\mathbb{R}^k$, and the functions $f$ and $\Psi$ are convex functions.
\end{assumption}

For each $u(\cdot)\in\mathcal{U}[0,T]$, there exists a unique strong solution $(X^u(\cdot),Y^u(\cdot))$ to the controlled SDEs (\ref{ode_1}) and (\ref{ode_2}), and the map $u(\cdot)\mapsto(X^u(\cdot),Y^u(\cdot))$ is continuous with respect to convergence in probability. The cost functional
\[
J(u(\cdot))=\mathbb{E}\left[\int_0^{\tau^u}f(X^u(t),u(t))\mathrm{d}t+\Psi(X^u(\tau^u))\right]
\]
is well-defined for all $u(\cdot)\in\mathcal{U}[0,T]$. We now state and prove the main existence theorem.
\begin{theorem}\label{thm:existence}
Let Assumptions \ref{assb-b2}, \ref{ass-fai} and \ref{ass-f-pai} hold. Then there exists an optimal control $\bar{u}(\cdot)\in\mathcal{U}[0,T]$ and an optimal pair $(\bar{u}(\cdot),\bar{X}(\cdot))$ such that
\[
J(\bar{u}(\cdot))=\inf_{u(\cdot)\in\mathcal{U}[0,T]}J(u(\cdot)).
\]
\end{theorem}

\begin{proof}
The proof proceeds in three steps: constructing a minimizing sequence, extracting a weakly convergent subsequence, and verifying that the weak limit is optimal.
Let $J^*=\inf_{u(\cdot)\in\mathcal{U}[0,T]}J(u(\cdot))$. Since $U$ is bounded (see Assumption \ref{ass-f-pai}), there exists a sequence $\{u_n(\cdot)\}_{n\geq1}\subset\mathcal{U}[0,T]$ such that
\[
\lim_{n\to\infty}J(u_n(\cdot))= J^*\quad \text{and} \quad \lim_{n\to\infty}\tau^{u_n}=\bar{\tau}.
\]

Since $U$ is bounded and convex, the sequence $\{u_n(\cdot)\}$ is bounded in $L^2_{\mathcal{F}}([0,T];U)$. Thus, there exists a subsequence (still denoted $\{u_n(\cdot)\}$) and a control $\bar{u}(\cdot)\in\mathcal{U}[0,T]$ such that
\[
u_n(\cdot)\rightharpoonup\bar{u}(\cdot)\quad\text{weakly in }L^2_{\mathcal{F}}([0,T];U).
\]
By Mazur's theorem, we can obtain a sequence
\[
\tilde{u}_j(\cdot):=\sum_{i\geq 1}a_{ij}u_{i+j}(\cdot),\ a_{ij}\geq 0,\ \sum_{i\geq 1}a_{ij}=1
\]
such that
\[
\tilde{u}_j(\cdot)\to \bar{u}(\cdot)\quad\text{strongly in }L^2_{\mathcal{F}}([0,T];U).
\]
By the continuity properties of controlled SDEs, the corresponding states satisfy
\[
X^{\tilde{u}_j}(\cdot)\to\bar{X}(\cdot)\quad\text{strongly in }L^2(\Omega;C([0,T];\mathbb{R}^m)),
\]
where $\bar{X}(\cdot)=X^{\bar{u}}(\cdot)$.

Note that $\mathbb{E}[Y^{{u}_j}(\cdot)]$ satisfies a linear equation, and $\lim_{j\to \infty}\tau^{u_j}=\bar{\tau}$, we have
$$
\mathbb{E}[Y^{\tilde{u}_j}(t)]=\sum_{i\geq 1}a_{ij}\mathbb{E}[Y^{{u}_{i+j}}(t)],
$$
and
$$
\lim_{j\to\infty}\tau^{\tilde{u}_j}=\tau^{\bar{u}}=\bar{\tau}.
$$
By the convexity of $f$ and $\Psi$, we obtain
\[
\begin{aligned}
J(\bar{u}(\cdot))&=\lim_{n\to\infty}J(\tilde{u}_j(\cdot))\\
&=\lim_{j\to\infty}\mathbb{E}\left[\int_0^{\tau^{\tilde{u}_j}}f(X^{\tilde{u}_j}(t),\tilde{u}_j(t))\mathrm{d}t
+\Psi(X^{\tilde{u}_j}(\tau^{\tilde{u}_j}))\right]\\
&\leq
\lim_{j\to\infty}\sum_{i\geq 1}a_{ij}\mathbb{E}\left[\int_0^{\tau^{\tilde{u}_j}}f(X^{{u}_{i+j}}(t),{u}_{i+j}(t))\mathrm{d}t
+\Psi(X^{{u}_{i+j}}(\tau^{\tilde{u}_j}))\right]\\
&=\lim_{j\to\infty}\sum_{i\geq 1}a_{ij}\mathbb{E}\left[\int_0^{\tau^{{u}_{i+j}}}f(X^{{u}_{i+j}}(t),{u}_{i+j}(t))\mathrm{d}t
+\Psi(X^{{u}_{i+j}}(\tau^{{u}_{i+j}}))\right]\\
&=J^*.
\end{aligned}
\]
Since $J^*$ is the minimum value of $\inf_{u(\cdot)\in\mathcal{U}[0,T]}J(u(\cdot))$, we conclude $J(\bar{u}(\cdot))=J^*$. Thus, $\bar{u}(\cdot)$ is an optimal control.
\end{proof}

\begin{remark}
The key novelties in this existence proof compared to classical fixed-horizon problems are: The continuity of the terminal time map $u(\cdot)\mapsto\tau^u$, which requires the mean-field target $\mathbb{E}[Y^u(t)]$ satisfying a linear equation; The existence theorem confirms that the optimal control problem is well-posed under the given conditions, providing a solid foundation for the subsequent analysis of necessary conditions and applications. Theorem \ref{thm:existence} confirms the existence of an optimal control for linear systems in Sections \ref{sec:time} and \ref{sec:app}.
\end{remark}

\section{Linear time-optimal control problem}\label{sec:time}
In this section, we consider a linear time-optimal control problem where the state $X^u(\cdot)$ satisfies a linear stochastic differential equation (\ref{e-line-ode_1}) and the cost functional is given by
\begin{equation}
J(u(\cdot))={\tau^{u}},\label{line-cost}%
\end{equation}
subject to the minimum time constraint
\begin{equation}
\label{line-time-1}
\tau^u=\inf\bigg{\{}t:\mathbb{E}[Y^u(t)]\leq 0,\ t\in [0,T] \bigg{\}}\wedge T,
\end{equation}
where $Y^u(t)$ satisfies equation (\ref{e-line-y-ode}).
Applying Theorem \ref{the-Max}, we have the following corollary.
\begin{corollary}\label{line-coro-1}
For the linear time-optimal control problem, the following holds:

 (i). If $\tau^{\bar{u}}<T$, then
\begin{equation}%
\begin{array}
[c]{ll}%
&-\displaystyle\frac{\hat{K}(t)(u-\bar{u}(t))}{G^{\bar{u}}
(\tau^{\bar{u}})}\leq 0,\\
\end{array}
\label{line-prin-2}%
\end{equation}
for any $u\in U$, $\text{a.e.}\ t \in[0,\tau^{\bar{u}}]$, $P-\text{a.s.}$

(ii). If $\bigg{\{}t:\mathbb{E}[Y^{\bar{u}}(t))]\leq 0,\ t\in [0,T] \bigg{\}}=\varnothing$, then
\begin{equation}%
\begin{array}
[c]{ll}%
&0\leq 0,\\
\end{array}
\end{equation}
for any $u\in U$, $\text{a.e.}\ t \in[0,\tau^{\bar{u}}]$, $P-\text{a.s.}$

(iii). If $\inf\bigg{\{}t:\mathbb{E}[Y^{\bar{u}}(t))]\leq 0,\ t\in [0,T] \bigg{\}}=T$, then
\begin{equation}\label{line-equa-smp-2}
\begin{array}
[c]{ll}%
&-\displaystyle\frac{\hat{K}(t)(u-\bar{u}(t))}{G^{\bar{u}}
(\tau^{\bar{u}})}\leq 0,\\
\end{array}
\end{equation}
or
\begin{equation}\label{line-equa-smp-3}
\begin{array}
[c]{ll}%
&0\leq 0,\\
\end{array}
\end{equation}
for any $u\in U$, $\text{a.e.}\ t \in[0,\tau^{\bar{u}}]$, $P-\text{a.s.}$
\end{corollary}

\begin{remark}
In Corollary \ref{line-coro-1}, we derive the stochastic maximum principle for the linear time-optimal control problem under three cases for $\tau^{\bar{u}}$. Case (i) yields new results. Case (ii) (empty set condition) give the trivial inequality $0\leq 0$ providing no meaningful information.
Case (iii) combines the previous cases. We therefore focus on case (i) in the subsequent analysis.
\end{remark}

In light of Corollary \ref{line-coro-1}, we focus on case (i) with $\tau^{\bar{u}} < T$. Assume $Y^u(t)$ satisfies
\begin{equation}\label{line-y-ode}
\left\{
\begin{aligned}
\text{d}{Y}^u(t)=&\big{[}E^{\top}_1\mathbb{E}[X^u(t)]+E^{\top}_2X^u(t)
+E^{\top}_3\mathbb{E}[AX^u(t)+Bu(t)]+E^{\top}_4u(t)\big{]}\text{d}t\\
&+g(\mathbb{E}[X^u(t)],X^u(t),u(t))\text{d}W(t)\\
Y^u(0)=&y_0,
\end{aligned}
\right.
\end{equation}
where $E^{\top}_1,E^{\top}_2,E^{\top}_3,E^{\top}_4$ are constant vectors.
\begin{theorem}\label{the-time-1}
For the linear time-optimal control problem, let $\tau^{\bar{u}}<T$, we have
\begin{equation}%
\begin{array}
[c]{ll}%
&\displaystyle \frac{{[}-p_0(t)^{\top}B+E^{\top}_3B+E^{\top}_4{]}(u-\bar{u}(t))}{(E^{\top}_1+E^{\top}_2+E^{\top}_3A)\mathbb{E}[\bar{X}(\tau^{\bar{u}})]
+(E^{\top}_3B+E^{\top}_4)\mathbb{E}[\bar{u}(\tau^{\bar{u}})]}\leq 0,\\
\end{array}
\label{line-prin-2}%
\end{equation}
for any $u\in U$, $\text{a.e.}\ t \in[0,\tau^{\bar{u}}]$, which is a necessary condition for a deterministic optimal control, where $p_0(t)$ satisfies the ordinary differential equation
\begin{equation*}
\begin{array}
[c]{rl}%
\text{d}{p}_0(t)^{\top}= & \bigg{[}-p_0(t)^{\top}A+E^{\top}_1+E^{\top}_2
                   +E^{\top}_3A\bigg{]}\text{d}t\\
p_0(\tau^{\bar{u}})^{\top}= &0,
\end{array}
\end{equation*}
and
$$
p_0(t)=-\int_t^{\tau^{\bar{u}}}(E^{\top}_1+E^{\top}_2+E^{\top}_3A)e^{A(s-t)}\text{d}s.
$$

Furthermore, if the control set $U$ is a bounded closed rectangle in $\mathbb{R}^k$, i.e.,
\[
U = [a_1, b_1] \times [a_2, b_2] \times \dots \times [a_k, b_k], \quad a_i \leq b_i,
\]
and
$$
(E^{\top}_1+E^{\top}_2+E^{\top}_3A)\mathbb{E}[\bar{X}(\tau^{\bar{u}})]
+(E^{\top}_3B+E^{\top}_4)\mathbb{E}[\bar{u}(\tau^{\bar{u}})]\neq 0,
$$
then the necessary condition (5.12) implies that the optimal control $\bar{u}(t) = (\bar{u}_1(t), \dots, \bar{u}_k(t))$ is of \emph{bang-bang} type for each component, except possibly on a set of singular times.
\end{theorem}

\begin{proof}
Note that
\begin{equation*}
K(x,x',b(x,u),u,p_0,q_0)=(Ax+Bu)^{\top}p_0+\sum_{j=1}^d(C_jx+D_ju)^{\top}q_0^j-h(x',x,b(x,u),u),
\end{equation*}
so
\begin{equation*}
\begin{array}
[c]{rl}%
\hat{K}(t)=&{K}_u(\bar{X}{(t)},\mathbb{E}[\bar{X}{(t)}],\mathbb{E}[b(\bar{X}{(t)},\bar{u}(t))],\bar{u}(t),
p_0(t),q_0(t))\\
&+\bar{h}_{x_3}(t)^{\top}
\mathbb{E}[b_u(\bar{X}(t),\bar{u}(t))]-\mathbb{E}[\bar{h}_{x_3}(t)^{\top}]
b_u(\bar{X}(t),\bar{u}(t))\\
=&p_0(t)^{\top}B+\sum_{j=1}^dq_0^j(t)^{\top}D_j-E_3B-E_4,
\end{array}
\end{equation*}
where $(p_0(t),q_0(t))$ satisfies
\begin{equation*}
\begin{array}
[c]{rl}%
-\text{d}{p}_0(t)= & \bigg{[}A^{\top}p_0(t)+ \displaystyle\sum_{j=1}^dC^{\top}_{j}q_0^j(t)
                   -E^{\top}_2-E^{\top}_1
                   -A^{\top}E^{\top}_3\bigg{]}\text{d}t\\
                   &-q_0(t)\text{d}W(t),\ t\in[0,\tau^{\bar{u}}),\\
p_0(\tau^{\bar{u}})= &0.
\end{array}
\end{equation*}
Thus, $q_0(t)=0$, and $p_0(t)$ satisfies
\begin{equation*}
\begin{array}
[c]{rl}%
\text{d}{p}_0(t)^{\top}= & \bigg{[}-p_0(t)^{\top}A+E^{\top}_1+E^{\top}_2
                   +E^{\top}_3A\bigg{]}\text{d}t\\
p_0(\tau^{\bar{u}})^{\top}= &0,
\end{array}
\end{equation*}
and
$$
-\hat{K}(t)=-p_0(t)^{\top}B+E^{\top}_3B+E^{\top}_4.
$$
Furthermore, we have
$$
G^{\bar{u}}(\tau^{\bar{u}})=(E^{\top}_1+E^{\top}_2+E^{\top}_3A)\mathbb{E}[\bar{X}(\tau^{\bar{u}})]
+(E^{\top}_3B+E^{\top}_4)\mathbb{E}[\bar{u}(\tau^{\bar{u}})],
$$
so
$$
-\displaystyle\frac{\hat{K}(t)(u-\bar{u}(t))}{G^{\bar{u}}(\tau^{\bar{u}})}
=\frac{{[}-p_0(t)^{\top}B+E^{\top}_3B+E^{\top}_4{]}(u-\bar{u}(t))}{(E^{\top}_1+E^{\top}_2
+E^{\top}_3A)\mathbb{E}[\bar{X}(\tau^{\bar{u}})]
+(E^{\top}_3B+E^{\top}_4)\mathbb{E}[\bar{u}(\tau^{\bar{u}})]}.
$$

Define the switching function row vector $S(t) \in \mathbb{R}^{1 \times k}$ by
\[
S(t) = \frac{-p_0(t)^\top B + E^{\top}_3 B + E^{\top}_4}
{(E^{\top}_1+E^{\top}_2+E^{\top}_3A)\mathbb{E}[\bar{X}(\tau^{\bar{u}})]
+(E^{\top}_3B+E^{\top}_4)\mathbb{E}[\bar{u}(\tau^{\bar{u}})]}.
\]
Then for each $i = 1, \dots, k$,
\[
\bar{u}_i(t) =
\begin{cases}
a_i, & \text{if } S_i(t) < 0, \\
b_i, & \text{if } S_i(t) > 0.
\end{cases}
\]

The function $p_0(t)$ is given explicitly by
\[
p_0(t) = -\int_t^{\tau} (E^{\top}_1 + E^{\top}_2 + E^{\top}_3 A) e^{A(s-t)}  ds,
\]
so each component $S_i(t)$ is a real analytic function of $t$ on $[0, \tau^{\bar{u}}]$, being a linear combination of integrals of matrix exponentials. Therefore, each $S_i(t)$ has at most finitely many zeros in $[0, \tau^{\bar{u}}]$ unless it is identically zero. Consequently, each control component $\bar{u}_i(t)$ switches between its minimum and maximum values a finite number of times, confirming the bang-bang property.
\end{proof}

\begin{example}
To verify Theorem \ref{the-time-1}, let $\tau^{\bar{u}}<T$, $x_0=0$, $m=d=1$, $A=a=B=b,C=c,D=d,E_1=e_1,E_2=e_2,E_3=e_3,E_4=e_4$, and $U=[u_{min},u_{max}]$, where $0<u_{min}$. By Theorem \ref{the-time-1}, we have
\begin{equation}\label{exam_equ_1}
\displaystyle \frac{(-p_0(t)b+e_3b+e_4)(u-\bar{u}(t))}{(e_1+e_2+e_3a)\mathbb{E}[\bar{X}(\tau^{\bar{u}})]
+(e_3b+e_4)\mathbb{E}[\bar{u}(\tau^{\bar{u}})]}\leq 0,
\end{equation}
and
$$
p_0(t)=\frac{1}{a}(e_1+e_2+e_3a)(1-e^{a(\tau^{\bar{u}}-t)}).
$$
The optimal state process $\bar{X}(\cdot)$ satisfies a linear equation, so
$$
\mathbb{E}[\bar{X}(\tau^{\bar{u}})]=\int_0^{\tau^{\bar{u}}}ae^{a(\tau^{\bar{u}}-s)}
\mathbb{E}[\bar{u}(s)]\text{d}s.
$$

Thus, the necessary condition (\ref{exam_equ_1}) becomes
\begin{equation}\label{exam_equ_2}
\displaystyle \frac{((e_1+e_2+e_3a)(e^{a(\tau^{\bar{u}}-t)}-1)+e_3b+e_4)(u-\bar{u}(t))}{(e_1+e_2+e_3a)\mathbb{E}[\bar{X}(\tau^{\bar{u}})]
+(e_3b+e_4)\mathbb{E}[\bar{u}(\tau^{\bar{u}})]}\leq 0, \quad u\in [u_{min},u_{max}].
\end{equation}

Now, we consider different value of parameters $a$, $(e_1+e_2+e_3a)$ and $(e_3b+e_4)$:

$(i)$. For $a>0$, $(e_1+e_2+e_3a)>0$, $(e_3b+e_4)>0$ or $a>0$, $(e_1+e_2+e_3a)<0$, $(e_3b+e_4)<0$, it follows that
$$
\frac{(e_1+e_2+e_3a)(e^{a(\tau^{\bar{u}}-t)}-1)+e_3b+e_4}{(e_1+e_2+e_3a)\mathbb{E}[\bar{X}(\tau^{\bar{u}})]
+(e_3b+e_4)\mathbb{E}[\bar{u}(\tau^{\bar{u}})]}>0,
$$
which implies
$$
u-\bar{u}(t)\leq 0,\quad u\in [u_{min},u_{max}].
$$
Thus $\bar{u}(t)=u_{max},\ t\in [0,\tau^{\bar{u}}]$.

$(ii)$. For $a>0$,
$$(e_1+e_2+e_3a)(e^{a\tau^{\bar{u}}}-1)u_{min}+(e_3b+e_4)u_{max}>0,$$
and $(e_3b+e_4)<0$, we have
$$
(e_1+e_2+e_3a)\mathbb{E}[\bar{X}(\tau^{\bar{u}})]
+(e_3b+e_4)\mathbb{E}[\bar{u}(\tau^{\bar{u}})]>0
$$
and there exists $t_0\in (0, \tau^{\bar{u}})$ such that
$$
(e_1+e_2+e_3a)(e^{a(\tau^{\bar{u}}-t_0)}-1)+e_3b+e_4=0.
$$
Combining these results, for $t\in [0,t_0)$, we have
$$
\frac{(e_1+e_2+e_3a)(e^{a(\tau^{\bar{u}}-t)}-1)+e_3b+e_4}{(e_1+e_2+e_3a)\mathbb{E}[\bar{X}(\tau^{\bar{u}})]
+(e_3b+e_4)\mathbb{E}[\bar{u}(\tau^{\bar{u}})]}>0,
$$
so $\bar{u}(t)=u_{max},\ t\in [0,t_0)$. For $t\in (t_0,\tau^{\bar{u}}]$, we have
$$
\frac{(e_1+e_2+e_3a)(e^{a(\tau^{\bar{u}}-t)}-1)+e_3b+e_4}{(e_1+e_2+e_3a)\mathbb{E}[\bar{X}(\tau^{\bar{u}})]
+(e_3b+e_4)\mathbb{E}[\bar{u}(\tau^{\bar{u}})]}<0,
$$
and $\bar{u}(t)=u_{min},\ t\in (t_0,\tau^{\bar{u}}]$.

\end{example}

\section{Application}\label{sec:app}
In this section, we apply the proposed framework to portfolio optimization with a target return and a risk constraint. We demonstrate how the unified framework naturally incorporates the dual objectives of minimizing both time and cost in such problems.

Consider an investor who aims to maximize the expected portfolio return while minimizing the time to reach a target wealth level \( \alpha^* \), which is a constant. Let \( X^u(\cdot) \) denote the portfolio value, and \( u(\cdot) \) the allocation to a risky asset, taking values in \( [\alpha^*, 1.25\alpha^*] \). The risk-free asset yields a return \( r \), and the stock price follows a geometric Brownian motion.

The state dynamics are given by
\[
\mathrm{d}X^u(t) = [r X^u(t) + (\mu - r) u(t)] \mathrm{d}t + \sigma u(t) \mathrm{d}W(t), \quad X^u(0) = x_0.
\]
The target process \( Y^u(t) \) is defined as the shortfall below the target,
\[
Y^u(t) = \alpha^* - X^u(t),
\]
and satisfies
\[
\mathrm{d}Y^u(t) = -[r X^u(t) + (\mu - r) u(t)] \mathrm{d}t - \sigma u(t) \mathrm{d}W(t), \quad Y^u(0) = \alpha^* - x_0.
\]

We define the terminal time as
\[
\tau^u = \inf \left\{ t : \mathbb{E}[Y^u(t)] \leq 0 \right\} \land T.
\]
The cost functional includes a quadratic control cost (e.g., transaction costs) prior to the terminal time \( \tau^u \):
\[
J(u(\cdot)) = \mathbb{E} \left[ \int_0^{\tau^u} (1+\frac{\lambda}{2} u(t)^2) \, \mathrm{d}t \right],
\]
where \( \lambda > 0 \).

Using the stochastic maximum principle from the unified framework, we have
\[
H(x,u,p,q) = [rx + (\mu - r) u] p + \sigma u q -1-\frac{\lambda}{2} u^2,
\]
\[
K(x,u,p_0,q_0) = [rx + (\mu - r) u] p_0 + \sigma u q_0 + rx + (\mu - r) u,
\]
from which it follows that
\[
H_u(x,u,p,q) = (\mu - r)p + \sigma q - \lambda u,
\]
\[
K_u(x,u,p_0,q_0) = (\mu - r)p_0 + \sigma q_0 + (\mu - r),
\]
\[
\hat{K} = (\mu - r)p_0 + \sigma q_0 + (\mu - r).
\]

The adjoint equations are
\[
\begin{cases}
-\mathrm{d}{p}_0(t) = \left[ r p_0(t) + r \right] \mathrm{d}t - q_0(t) \mathrm{d}W(t), \\
p_0(\tau^{\bar{u}}) = 0,
\end{cases}
\]
and
\[
\begin{cases}
-\mathrm{d}p(t) = r p(t) \mathrm{d}t - q(t) \mathrm{d}W(t), \\
p(\tau^{\bar{u}}) = 0.
\end{cases}
\]
The solutions are
\[
(p_0(t), q_0(t)) = \left( e^{r(\tau^{\bar{u}} - t)} - 1, 0 \right), \quad (p(t), q(t)) = (0, 0).
\]
Thus,
\[
H_u(\bar{X}(t), \bar{u}(t), p(t), q(t)) = -\lambda \bar{u},
\]
\[
\hat{K}(t) = (\mu - r) e^{r(\tau^{\bar{u}} - t)}.
\]

We now apply Theorem \ref{the-Max} to determine the optimal control. Note that
\[
G^{\bar{u}}(\tau^{\bar{u}}) = -\left[ r \mathbb{E}[\bar{X}(\tau^{\bar{u}})] + (\mu - r) \mathbb{E}[\bar{u}(\tau^{\bar{u}})] \right].
\]
Assuming \( \tau^{\bar{u}} < T \) (where \( T \) is sufficiently large), condition (i) of Theorem \ref{the-Max} gives
\begin{equation}
\begin{aligned}
-\lambda \bar{u}(t)(u - \bar{u}(t)) + \frac{
\mathbb{E}\left[1+\frac{\lambda}{2} \bar{u}(\tau^{\bar{u}})^2 \right] (\mu - r) e^{r(\tau^{\bar{u}} - t)}(u - \bar{u}(t))}
{r \mathbb{E}[\bar{X}(\tau^{\bar{u}})] + (\mu - r) \mathbb{E}[\bar{u}(\tau^{\bar{u}})]} \leq 0.
\end{aligned}
\end{equation}
Hence,
\[
2 \bar{u}(t)(u - \bar{u}(t)) - \frac{
\mathbb{E}[\frac{2}{\lambda}+\bar{u}(\tau^{\bar{u}})^2] (\mu - r) e^{r(\tau^{\bar{u}} - t)}(u - \bar{u}(t))}
{r \mathbb{E}[\bar{X}(\tau^{\bar{u}})] + (\mu - r) \mathbb{E}[\bar{u}(\tau^{\bar{u}})]} \geq 0,
\]
and thus
\begin{equation}\label{exam1_equ0}
\left[ 2 \bar{u}(t) - \frac{
\mathbb{E}[\frac{2}{\lambda}+\bar{u}(\tau^{\bar{u}})^2] (\mu - r) e^{r(\tau^{\bar{u}} - t)}}
{r \mathbb{E}[\bar{X}(\tau^{\bar{u}})] + (\mu - r) \mathbb{E}[\bar{u}(\tau^{\bar{u}})]} \right] (u - \bar{u}(t)) \geq 0.
\end{equation}
This implies a deterministic control \( \bar{u}(\cdot) \), where \( \tau^{\bar{u}} \) satisfies
\[
\mathbb{E}[\bar{X}(\tau^{\bar{u}})] = \alpha^*, \quad \mathbb{E}[\bar{X}(t)] < \alpha^*, \quad t < \tau^{\bar{u}}.
\]

From inequality (\ref{exam1_equ0}), for any \( u \in [\alpha^*, 1.25\alpha^*] \) and \( t \in [0, \tau^{\bar{u}}] \), we have
\[
I_1 \leq 0, \ I_2 \leq 0 \quad \text{or} \quad I_1 \geq 0, \ I_2 \geq 0,
\]
where
\[
\begin{cases}
\displaystyle I_1 = \left[ 2 \bar{u}(t) - \frac{
\mathbb{E}[\frac{2}{\lambda}+\bar{u}(\tau^{\bar{u}})^2] (\mu - r) e^{r(\tau^{\bar{u}} - t)}}
{r \alpha^* + (\mu - r) \mathbb{E}[\bar{u}(\tau^{\bar{u}})]} \right], \\
I_2 = (u - \bar{u}(t)).
\end{cases}
\]
Note that the second term in \( I_1 \) is a decreasing function of \( t \). Let $\lambda=\frac{1}{\beta{\alpha^*}^2},\ \beta>0$. The optimal control \( \bar{u}(\cdot) \) takes the form
\begin{equation}\label{exam1_optm}
\begin{cases}
\displaystyle
\bar{u}(t) = 1.25\alpha^*, & t \in [0, t_1], \\
\displaystyle \bar{u}(t) = \frac{(2\beta+1)(\mu - r) e^{r(\tau^{\bar{u}} - t)}}{2\mu} \alpha^*, & t \in (t_1, t_2], \\
\bar{u}(t) = \alpha^*, & t \in (t_2, \tau^{\bar{u}}],
\end{cases}
\end{equation}
where $0\leq t_1\leq t_2\leq \tau^{\bar{u}}$, with \( t_1 \) and \( t_2 \) given by
\[
t_1 = \tau^{\bar{u}} - \frac{1}{r} \ln \frac{2.5\mu}{(2\beta+1)(\mu - r)}, \quad
t_2 = \tau^{\bar{u}} - \frac{1}{r} \ln \frac{2\mu}{(2\beta+1)(\mu - r)}.
\]
From \( \mathbb{E}[\bar{X}(\tau^{\bar{u}})] = \alpha^* \), \( \tau^{\bar{u}} \) satisfies
\[
\alpha^* = x_0 e^{r \tau^{\bar{u}}} + \int_0^{\tau^{\bar{u}}} (\mu - r) \bar{u}(t) e^{r(\tau^{\bar{u}} - s)} \, \mathrm{d}s.
\]

To obtain the solution, we combine the expressions for \( \bar{u}(\cdot), t_1, t_2, \) and \( \tau^{\bar{u}} \). Substituting the optimal control \( \bar{u}(t) \) into the terminal wealth condition yields
\begin{equation*}
\begin{array}
[c]{rl}%
\alpha^*=& \displaystyle x_0 e^{r\tau^{\bar{u}}} + (\mu - r)\alpha^* \bigg[ \int_0^{t_1} 1.25 e^{r(\tau^{\bar{u}}-t)} \mathrm{d}t
+ \int_{t_1}^{t_2} \frac{(2\beta+1)(\mu - r) e^{2r(\tau^{\bar{u}}-t)}}{2\mu}\mathrm{d}t\\
&\displaystyle \qquad \qquad\qquad\qquad+ \int_{t_2}^{\tau^{\bar{u}}} e^{r(\tau^{\bar{u}} - t)} \mathrm{d}t \bigg].
\end{array}
\end{equation*}
Let \( \tau^{\bar{u}} \) be the unknown terminal time. The switching times $t_1,t_2$ are
\begin{align*}
t_1 &= \tau^{\bar{u}} - \frac{1}{r} \ln \frac{2.5\mu}{(2\beta+1)(\mu - r)}, \\
t_2 &= \tau^{\bar{u}} - \frac{1}{r} \ln \frac{2\mu}{(2\beta+1)(\mu - r)}.
\end{align*}
Computing the integrals explicitly, we obtain
\begin{equation}\label{exam1_solv}
\alpha^* = x_0 e^{r\tau^{\bar{u}}} + (\mu - r) \alpha^* \left( L_1 + L_2 + L_3 \right),
\end{equation}
where
\[
\begin{cases}
\displaystyle L_1 = 1.25 \frac{e^{r\tau^{\bar{u}}} - e^{r(\tau^{\bar{u}} - t_1)}}{r}, \\
\displaystyle L_2 = \frac{(2\beta+1)(\mu - r)}{4\mu r} \left( e^{2r(\tau^{\bar{u}} - t_1)} - e^{2r(\tau^{\bar{u}} - t_2)} \right), \\
\displaystyle L_3 = \frac{e^{r(\tau^{\bar{u}} - t_2)}-1}{r}.
\end{cases}
\]

We solve the nonlinear equation (\ref{exam1_solv}) numerically using the parameter values,
\[
r = 0.05, \quad \mu = 0.10, \quad \beta=\frac{\mu}{\mu-r}-0.8, \quad \alpha^* = 10, \quad x_0 = 1.
\]
The solution is
\[
\tau^{\bar{u}} = 10.92, \quad t_1 = 3.21, \quad t_2 = 7.67,
\]
and the optimal control is shown in Figure \ref{fig_1}. The optimal control is given by
\[
\begin{cases}
\displaystyle
\bar{u}(t) = 12.5, & t \in [0, 3.21], \\
\displaystyle \bar{u}(t) = 8.5 e^{0.05(10.92 - t)}, & t \in (3.21, 7.67], \\
\bar{u}(t) = 10, & t \in (7.67, 10.92].
\end{cases}
\]

\begin{figure}[H]
	\centering
	\includegraphics[width=0.75\textwidth]{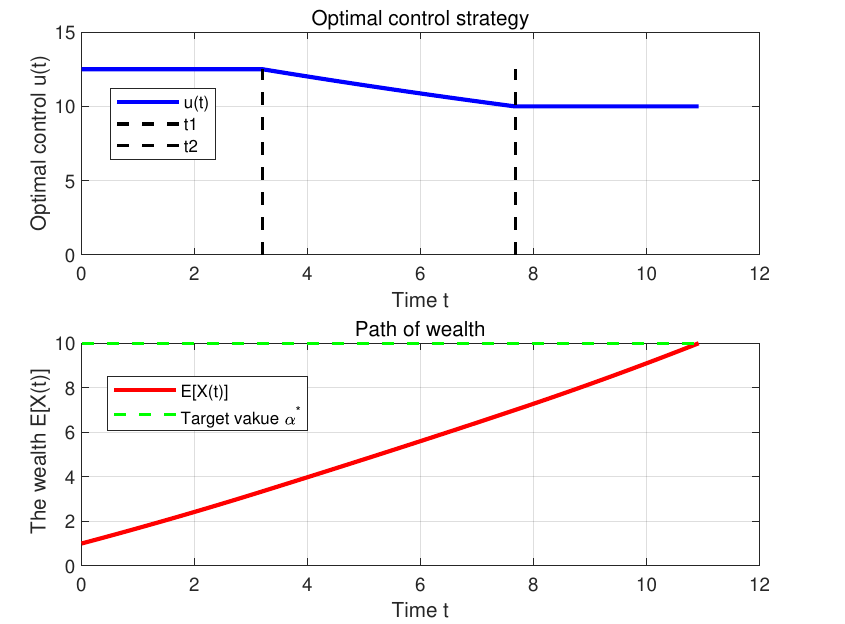}
	\caption{Optimal control and optimal wealth}
	\label{fig_1}
\end{figure}

Simulations indicate that the optimal strategy promotes rapid growth initially and reduces risk as the target is approached. Furthermore, it is straightforward to verify that the optimal control would be constant if only time-optimal or classical control structures were considered. This example illustrates how the unified stochastic optimal control framework can be applied to real-world financial problems.

\section{Conclusion}\label{sec:con}
In this study, we investigate a novel stochastic optimal control structure that unified time-optimal control problems and classical stochastic optimal control problems within a single framework. In this structure, the terminal time $\tau^u$ is a varying deterministic functional of control $u(\cdot)$, and $Y^u(\cdot)$ describes the performance or characteristics of an observable target. This allow us to simultaneously balance minimizing the varying terminal time $\tau^u$ and the cost functional $J(u(\cdot))$.

This paper presents a detailed framework for unified stochastic optimal control and rigorously establishes the corresponding stochastic maximum principle. Subsequently, we delve into the time-optimal control problem within this novel framework. Future work should aim to establish general sufficient conditions for optimality to complement the necessary conditions provided by the maximum principle. Additionally, developing computationally efficient numerical methods is crucial for solving the complex coupled forward-backward stochastic differential equations that arise from this framework. Further research could also explore extending the dynamic programming principle to this unified control setting and investigating its applications in areas such as mathematical finance and engineering systems.

\section*{Acknowledgements}
Funding: This work was supported by the National Natural Science Foundation of China (Grant No.12471450) and Taishan Scholar Talent Project Youth Project. Conflict of Interest: The authors declare that they have no conflict of interest.

\end{document}